\documentclass[aop,MSNbibl,seceqn,dvips]{arximspdf}
\usepackage{accents}
\usepackage{mathbh}

\doi{10.1214/13-AOP869} 
\volume{42}
\issue{2}
\pubyear{2014}
\firstpage{818}
\lastpage{864}

\makeatletter
\newcommand{\rrVert}{\Vert}
\newcommand{\rrvert}{\vert}
\newcommand{\llVert}{\Vert}
\newcommand{\llvert}{\vert}
\renewcommand{\mathring}[1]{\accentset{\circ}{#1}}
\newcommand{\eqref}[1]{(\ref{#1})}

\newcommand{\mcB}{\mathcal{B}}

\newproclaim{definition}{Definition}[section]
\newproclaim{remark}[definition]{Remark}
\newtheorem{lemma}[definition]{Lemma}
\newproclaim{example}[definition]{Example}
\newtheorem{proposition}[definition]{Proposition}
\newtheorem{theorem}[definition]{Theorem}
\newtheorem{corollary}[definition]{Corollary}

\newcommand{\td}{\tilde}
\newcommand{\z}{\zeta}
\renewcommand{\P}{\mathbb{P}}
\renewcommand{\b}{\beta}
\renewcommand{\k}{\kappa}
\renewcommand{\l}{\lambda}
\renewcommand{\O}{\Omega}
\renewcommand{\t}{\theta}
\renewcommand{\div}{\operatorname{div}}

\newcommand{\mcO}{\mathcal{O}}
\newcommand{\mcF}{\mathcal{F}}
\newcommand{\mcK}{\mathcal{K}}
\newcommand{\mcP}{\mathcal{P}}
\newcommand{\mcD}{\mathcal D}

\newcommand{\R}{\mathbb{R}}
\newcommand{\N}{\mathbb{N}}

\newcommand{\D}{\Delta}
\newcommand{\g}{\gamma}
\newcommand{\s}{\sigma}
\newcommand{\vp}{\varphi}
\newcommand{\ve}{\varepsilon}

\makeatother

\begin{document}
\begin{frontmatter}

\title{Random attractors for stochastic porous media equations
perturbed by space--time linear multiplicative noise}
\runtitle{Random attractors for SPME}

\begin{aug}
\author[A]{\fnms{Benjamin} \snm{Gess}\corref{}\ead[label=e1]{gess@math.tu-berlin.de}\thanksref{t1}}
\thankstext{t1}{Supported by DFG-International Graduate College
``Stochastics and Real World Models,'' the SFB-701 and the
BiBoS-Research Center.}
\runauthor{B. Gess}
\affiliation{Humboldt-Universit\"at zu Berlin}
\address[A]{Institut f\"ur Mathematik\\
Humboldt-Universit\"at zu Berlin\\
Rudower Chaussee 25 \\
12489 Berlin\\
Germany\\
\printead{e1}} 
\end{aug}

\received{\smonth{2} \syear{2012}}
\revised{\smonth{4} \syear{2013}}

%
\begin{abstract}
Unique existence of solutions to porous media equations driven by
continuous linear
multiplicative space--time rough signals is proven for initial data in
$L^1(\mcO)$ on bounded domains $\mcO$.
The generation of a continuous, order-preserving random dynamical
system on $L^1(\mcO)$ and the existence of a~\mbox{random} attractor
for stochastic porous media equations perturbed by
linear multiplicative noise in \textit{space and time}
is obtained. The random attractor is shown to be compact and attracting
in $L^\infty(\mcO)$ norm. Uniform $L^\infty$ bounds
and uniform space--time continuity of the solutions is shown. General
noise including fractional Brownian motion for all Hurst
parameters is treated and a pathwise Wong--Zakai result for driving
noise given by a continuous semimartingale is obtained.
For fast diffusion equations driven by continuous linear multiplicative
space--time rough signals, existence of solutions is
proven for initial data in $L^{m+1}(\mcO)$.
\end{abstract}



\end{frontmatter}

%
%

\section{Introduction}\label{sec1}

The qualitative study of stochastic dynamics induced by stochastic
partial differential equations (SPDE) especially in the case of
non-Markovian noise is based on the theory of random dynamical systems
(RDS); cf., for example, \cite{A98}. Since the foundational work \cite
{S92,CF94,CDF97} the long-time behavior of several quasilinear SPDE has
been investigated by means of the existence of random attractors.
However, all these results are restricted to simple models of the noise
(e.g., additive or real multiplicative\setcounter{footnote}{1}\footnote{By this we mean
multiplicative noise with diffusion coefficients independent of the
spatial variable.}) not including the important case of linear
multiplicative space--time noise. This is mainly due to the difficulty
to even define an associated RDS for more general SPDE. The generation
of an RDS is usually shown by use of a transformation of the SPDE into
a random PDE. Depending on the structure of the noise monotonicity and
coercivity properties of the drift are preserved under this
transformation. For example, this is the case for additive, real linear
multiplicative and for certain transport noise \cite{FL04}. For linear
multiplicative \textit{space--time} noise, however, this is not the case,
thus making the analysis of the random PDE much harder. The generation
of an RDS and the existence of random attractors for stochastic porous
media equations (SPME) with \textit{additive} noise has been obtained in
\cite{BGLR10,GLR11}. A~first approach to tackle the generation of an
RDS for SPME with linear multiplicative space--time noise, that is, for
equations of the form
%
\begin{equation}
\label{eqnspme} d X_t = \D \bigl( |X_t|^m
\operatorname{sgn}(X_t) \bigr)\,dt + \sum_{k=1}^N
\mu_k e_k X_t d\b^{(k)}_t,\qquad
1 < m < \infty,
\end{equation}
has been given in \cite{BR11b} by proving the unique existence of
pathwise solutions to a corresponding random PDE for essentially
bounded initial conditions $x \in L^\infty(\mcO)$. The existence and
uniqueness of probabilistically strong solutions to \eqref{eqnspme},
even including $0<m<1$ and all initial conditions $x \in
(H_0^1(\mcO) )^*$, has been obtained in~\cite{RRW07}. However,
this does not yield the existence of an RDS. The pathwise solutions to
the transformed equation constructed in \cite{BR11b} form an RDS $\vp$
on~$L^\infty(\mcO)$. However, neither continuity of $x \mapsto\vp
(t,\omega
)x$ nor continuity of $t \mapsto\vp(t,\omega)x$ has been obtained. These
properties of RDS are crucial to obtain the existence of random
attractors. Due to the strong norm on the state space $L^\infty(\mcO)$
especially the continuity in the initial condition is not clear.
In this paper we prove the generation of an RDS corresponding to SPME
driven by multiplicative space--time rough signals for all initial
conditions $X_0 \in L^1(\mcO)$, that is to equations of the form
%
\begin{eqnarray}
\label{eqnroughPDE} d X_t &=& \D \bigl(|X_t|^m
\operatorname{sgn}(X_t) \bigr)\,dt + \sum_{k=1}^N
f_k X_t \circ dz^{(k)}_t\qquad \mbox{on
} \mcO_T,
\nonumber
\\[-8pt]
\\[-8pt]
\nonumber
X(0) &=& X_0\qquad \mbox{on } \mcO,
\end{eqnarray}
with $1<m<\infty$, homogeneous Dirichlet boundary conditions, rough
driving signals $z^{(k)} \in C([0,T];\R)$ and with $f_k \in C^\infty
(\bar\mcO)$. We assume the number of signals $N$ to be finite and high
regularity for $f_k$ for simplicity only. In fact, most of the proofs
only require $\sum_{k=1}^\infty f_k(\xi)z_t^{k} \in C([0,T];C^2(\bar
\mcO
))$. The stochastic Stratonovich integral occurring in \eqref
{eqnroughPDE} is informal, and the rigorous justification of this
notation is part of our results. The resulting stochastic flow is
proven to be an RDS $\vp$ on $L^1(\mcO)$ which is continuous in the
initial condition and in time. Generalizing the notion of
quasi-continuity of RDS we show that $\vp$ is quasi-weakly-continuous
on $L^p(\mcO)$ for all $p \in[1,\infty)$ and
quasi-weakly$^*$-continuous on $L^\infty(\mcO)$. Moreover, we prove the
existence of an absorbing random set $F \subseteq X$ which even is
bounded in $L^\infty(\mcO)$, as well as asymptotic compactness of
$\vp$
on each $L^p(\mcO)$, $p \in[1,\infty]$ (requiring a uniform convexity
condition for $\mcO$ if $p = \infty$). Generalizing an existence result
for random attractors of quasi-continuous RDS, we deduce the existence
of a random attractor $A$ for $\vp$ [as an RDS on $L^1(\mcO)$], which
is compact and attracting in each $L^p(\mcO)$ with $p \in[1,\infty]$.
For \textit{semilinear} SPDE with linear multiplicative space--time noise,
a construction of stochastic flows and invariant manifolds can be found
in \cite{MZZ08}.

We obtain new spatial and temporal regularity properties for solutions
to \eqref{eqnroughPDE} analogous to those proved for deterministic
porous media type equations by De Giorgi--Nash--Moser type iteration
techniques in \cite{DB83}. More precisely, we prove that the solution
$X$ is locally equicontinuous on $\mcO_T$ (i.e., continuous on each
compact set $K \subseteq(0,T] \times\mcO$ with a modulus of
continuity independent of the initial condition). Under appropriate
assumptions on the boundary $\partial\mcO$ (on the initial data $X_0$,
resp.), also equicontinuity up to the boundary (continuity up to
initial time $t = 0$, resp.) is obtained. Applied to driving signals
given by independent Brownian motions this implies a new regularity
result for the variational stochastic solution $X$ corresponding to
\eqref{eqnroughPDE}, namely $\P$-a.s. local equicontinuity on~$\mcO
_T$. This complements the regularity results given in \cite{G12}, where
it is shown that $|X|^m\operatorname{sgn}(X_t) \in L^2([0,T] \times\O;H_0^1(\mcO))$
and $X \in L^\infty([0,T];L^{m+1}(\O\times\mcO))$ if the initial
condition is regular enough.

We consider \eqref{eqnroughPDE} driven by rough signals $z^{(k)} \in
C([0,T];\R)$ which do not necessarily need to be given as paths of a
continuous semimartingale. The construction proceeds by a Wong--Zakai
approximation of the driving noise, proving the existence of a limit
solution independent of the chosen approximating sequence. This is
reminiscent of the rough paths approach to SPDE developed in \cite
{LS00,LS02,BM01,BM01-2,GT10,H11,CFO11}, where SPDE driven by rough
paths are treated by transformation to a random PDE; cf. \eqref
{eqnroughPDEtransformed} below. We note, however, that due to the
form of perturbation considered in this paper, no rough paths
techniques such as rough paths topologies on augmented paths spaces are
required. If the driving signal is given by a continuous
semimartingale, we prove that this limit solution solves the
corresponding SPDE, and we thereby obtain a pathwise Wong--Zakai result
for SPME driven by linear multiplicative space--time semimartingale noise.

The long-time behavior of SPDE can be analyzed in terms of the
associated Markovian semigroup and its ergodicity or in terms of the
associated RDS and its random attractor. As soon as the driving noise
lacks the Markov property, the SPDE does not induce a Markovian
semigroup anymore. In contrast, analyzing the associated RDS merely
requires the noise to have stationary increments and some path
regularity; cf., for example, \cite{GLR11}. In particular, RDS can be
used to study long-time behavior of SPDE driven by fractional Brownian
Motion (fBm). The characteristic long-range dependence of fBm makes an
investigation of the induced stochastic dynamics especially intriguing.
In this paper we only assume that the noise has stationary increments
and continuous paths, thus including fBm for all Hurst parameters.

Concerning the theory of RDS and random attractors, we slightly
generalize the concept of quasi-weakly continuous RDS in order to show
that the constructed RDS is quasi-weakly$^*$-continuous on $L^\infty
(\mcO
)$ and has a random attractor with respect to the $L^\infty$-norm.
Since the equicontinuity of the solutions only holds locally, that is,
on every compact set $K \subseteq\mcO$, the notion of compact
absorption has to be replaced by asymptotic compactness (cf., e.g., \cite{BL06} and the references therein), a~concept usually needed in
order to treat unbounded domains. This application of asymptotic
compactness seems to be new.

Our methods to prove the existence of solutions to \eqref
{eqnroughPDE} for initial conditions $X_0 \in L^{m+1}(\mcO)$ also
apply in the case of fast diffusions (i.e., for $0<m<1$) driven by
continuous signals. In particular, this generalizes results given in
\cite{BR11b} since no restrictions on the dimension $d$ nor on the
exponent $0<m<1$ are assumed. In order not to overload the
presentation, the case of fast diffusion equations is treated as a
remark only (Remark \ref{rmkfde}, below).

SPME and stochastic fast diffusion equations (SFDE) have been
intensively investigated in recent years; cf., for example, \cite
{DPR04,K06,DPRRW06,RRW07,RW08,BDPR08,BDPR08-2,BDPR09,G12} and
references therein. The long-time behavior of SPME with Brownian
additive noise in terms of the existence of a random attractor has
first been treated in \cite{BGLR10} which then has been partially
extended to more generally distributed additive noise in \cite
{GLR11,G13}. The SFDE ($0 < m < 1$) with linear multiplicative
space--time noise has been first solved in \cite{RRW07,G13}.
Subsequently, extinction in finite time with positive probability has
been shown in \cite{BDPR09-3} and in a more singular case which is used
as model to study self-organized criticality in \cite{BDPR09-2}.

A concise announcement of the results presented here has appeared in~\cite{G11c-2}.

\subsection{Survey of the construction of the RDS and of the proofs of
its properties}
Let $1 < m < \infty$ and $\Phi\in C(\R)$ be given by
\[
\Phi(r):= |r|^m \operatorname{sgn}(r).
\]
\textit{First part}: In the first part we construct ``pathwise''
solutions to the rough partial differential equation \eqref
{eqnroughPDE}. Step by step we will allow rougher signals $z^{(k)}$
and initial conditions $X_0$ at the expense of weaker notions of
solutions. The construction of solutions to \eqref{eqnroughPDE} for
signals of bounded variation proceeds by first transforming the
equation into a PDE and then constructing solutions to this transformed
equation. Let
%
\begin{equation}
\label{eqnmudef} \mu_t(\xi):= -\sum_{k=1}^N
f_k(\xi) z^{(k)}_t.
\end{equation}
Defining $Y = e^{\mu}X$, we obtain the transformed equation (which was
first studied in~\cite{BDPR09-2,BR11b})
%
\begin{eqnarray}
\label{eqnroughPDEtransformed} \partial_t Y_t &=&
e^{\mu_t} \D\Phi\bigl(e^{-\mu_t} Y_t\bigr)\qquad \mbox{on }
\mcO _T,
\nonumber
\\[-8pt]
\\[-8pt]
\nonumber
Y(0) &=& Y_0\qquad \mbox{on } \mcO,
\end{eqnarray}
with homogeneous Dirichlet boundary conditions. This transformation is
rigorous for driving signals of bounded variation (Theorem \ref
{thmtransformationbddvarn}) as well as for signals given by paths of
a continuous semimartingale (Theorem \ref
{thmtransformationveryweaksemimartingale}). Next, we prove
uniqueness of essentially bounded solutions to \eqref
{eqnroughPDEtransformed} (Theorem \ref
{thmuniquenessveryweaksoln}). Equation \eqref
{eqnroughPDEtransformed} does not fall into any class of equations
for which unique existence of solutions is known. In particular, \eqref
{eqnroughPDEtransformed} does not satisfy the structural assumptions
required in \cite{A91,DK87,G89,S83b}.

For continuous driving signals, we construct weak solutions to \eqref
{eqnroughPDEtransformed} as limits of solutions to a nondegenerate,
smooth approximation; that is, we approximate $\Phi$ by $\Phi
^{(\delta)}$
with $0 < C(\delta) \le\dot\Phi^{(\delta)}$ and the signal $z$ by
$z^{(\delta)}
\in C^\infty([0,T];\R^N)$. Solutions to these nondegenerate
approximations are obtained via classical existence results for
quasilinear equations. Passing to the limit $\delta\to0$ in order to
obtain weak solutions to \eqref{eqnroughPDEtransformed} requires
uniform $L^\infty(\mcO)$ bounds for the approximating solutions~$Y^{(\delta
)}$. Such bounds are obtained by constructing bounded supersolutions to
\eqref{eqnroughPDEtransformed}. Thereby, the existence of \textit
{weak solutions} to \eqref{eqnroughPDEtransformed} satisfying
analogous $L^\infty(\mcO)$ bounds is obtained for essentially bounded
initial conditions (Theorem~\ref{thmexistenceweak}). In case of
signals of bounded variation this yields weak solutions to \eqref
{eqnroughPDE} by transformation.

Next, we approximate general continuous driving signals $z$ by
continuous signals of bounded variation $z^{(\ve)}$ and prove that the
corresponding weak solutions $X^{(\ve)}$ converge to a limit $X$
independent of the chosen approximating sequence $z^{(\ve)}$. We call
the limit $X$ a \textit{rough weak solution} to \eqref{eqnroughPDE}
and observe $X=e^{-\mu}Y$, where $Y$ denotes the weak solution to
\eqref
{eqnroughPDEtransformed} for the continuous driving signal $z$
(Theorem \ref{thmroughPDE}).

In order to construct solutions for general initial data $X_0 \in
L^1(\mcO)$ we prove Lipschitz continuity of $X$ in the initial
condition with respect to the $L^1(\mcO)$ norm. For $X_0 \in L^1(\mcO)$
solutions are then obtained as \textit{limit solutions} by
approximation of $X_0$ by essentially bounded initial conditions
(Theorem \ref{thmlimitsoln}). Using an $L^1-L^\infty$ regularizing
property of the flow, these limit solutions are characterized as unique
\textit{generalized weak solutions} to \eqref
{eqnroughPDEtransformed} (Theorem \ref{thmcharaclimitsoln}). This
regularization property also builds the foundation of the proof of
bounded absorption. The key idea is to combine an interval splitting
technique as used in \cite{BR11b}, Lemma 3.3, with the known uniform
supersolution (i.e., a supersolution independent of the initial
condition) from the deterministic case
\[
U(t):= A t^{-{1}/{(m-1)}}\bigl(R^2-|\xi|^2
\bigr)^{{1}/{m}}.
\]
Combining these ideas we construct a new uniform supersolution for
\eqref{eqnroughPDEtransformed}. The resulting construction is quite
different from the one given in \cite{BR11b}.

Based on continuity results presented in \cite{DB83} we then prove that
the limit solutions are uniformly continuous on each compact set $K
\subseteq(0,T] \times\mcO$ (Theorem~\ref{thmctnlimitsoln}). This
finishes the treatment of the pathwise case.

\textit{Second part}: In the second part we consider SPME driven by
signals given as paths of stochastic processes
\begin{eqnarray*}
\label{eqnSPDE} d X_t &=& \D\Phi(X_t)\,dt + \sum
_{k=1}^N f_k X_t \circ
dz^{(k)}_t\qquad \mbox {on } \mcO_T,
\\
X(0) &=& X_0\qquad\mbox{on } \mcO,
\end{eqnarray*}
with homogeneous Dirichlet boundary conditions, where $z$ is an $\R
^N$-valued stochastic process with stationary increments and continuous
paths. Defining
\[
\vp(t,\omega)x = X(t,0;\omega)x
\]
yields an order-preserving RDS on $L^1(\mcO)$ (Theorem \ref
{thmgenerationRDS}), where $X(t,0;\omega)x$ is the solution obtained in
the first part driven by the signal $z =z(\omega)$. The uniform
$L^\infty
(\mcO)$ bound and the regularity results obtained for the rough PDE
\eqref{eqnroughPDE} continue to hold for $\vp$, which induces
asymptotic compactness of $\vp$ in each $L^p(\mcO)$, $p \in[1,\infty
]$. The existence of a random attractor in $L^1(\mcO)$ follows. In
order to deduce the attraction property in higher $L^p$-norms and in
particular with respect to the $L^\infty$-norm, we slightly generalize
the notion of quasi-weakly-continuous RDS for not necessarily reflexive
subspaces and thereby obtain the existence of a random attractor with
respect to every $L^p$-norm, $p \in[1,\infty]$ (Theorem \ref
{thmexistenceRA}).

In Section~\ref{secmainresult} we introduce the detailed setup and
present the main results. Proofs of the pathwise results are given in
Section~\ref{secroughcase} while the ones for the stochastic case and
the RDS $\vp$ are given in Section~\ref{secRDS}.

As usual in probability theory we denote the time-dependency of
functions by a subscript $X_t$ rather than by $X(t)$ in order to keep
the equations at a bearable length. We would like to apologize to the
readers with a more analytical background for this maybe unfamiliar notation.

\section{Setup and main results}\label{secmainresult}
Let $\mcO\subseteq\R^d$ be a smooth, bounded domain, $T > 0$ and
$\mcO
_T:= [0,T] \times\mcO$. By $\mcP\mcO_{[s,t]}$ we denote the
(time-inverted) parabolic boundary $[s,t] \times\partial\mcO \cup \{
t\} \times\mcO$, and we set $\mcP\mcO_T:= \mcP\mcO_{[0,T]}$. Let
$C(\mcO)$ be the set of continuous functions on $\mcO$, $C^{m,n}(\bar
\mcO_T) \subseteq C(\bar\mcO_T)$ be the set of all continuous functions
on $\mcO_T$ having $m$ continuous derivatives in time and $n$
continuous derivatives in space. By $C^{1-\operatorname{var}}([0,T];H)$ we denote the
set of all continuous functions of bounded variation and by
$C^w([0,T];H)$ the weakly continuous functions taking values in $H$. As
usual, $W^{m,p}(\mcO)$ denotes the Sobolev space of order $m$ in
$L^p(\mcO)$, $W_0^{m,p}(\mcO)$ the subspace of functions vanishing on
$\partial\mcO$, and we set $H_0^1(\mcO):= W_0^{1,2}(\mcO)$,
$H:=(H_0^1(\mcO))^*$. For a subset $K$ of a Banach space $X$ we define
$\|K\|_X:= \sup_{k \in K} \|k\|_X$.

\subsection{Porous medium equation driven by rough signals}\label{ssecRPME}

Let us first define what we mean by a solution to~\eqref{eqnroughPDE}
and~\eqref{eqnroughPDEtransformed}. Setting $B(x)(z):= \sum_{k=1}^N
f_k x z^{(k)}$ for $x \in L^1(\mcO)$ and $z \in\R^N$ we can rewrite
\[
B(X_t) \circ dz_t = \sum_{k=1}^N
f_k X_t \circ dz^{(k)}_t.
\]
As outlined in the \hyperref[sec1]{Introduction}, we will introduce several notions of
solutions to~\eqref{eqnroughPDE} and \eqref
{eqnroughPDEtransformed}, corresponding to the intermediate steps in
the construction of the solution for initial values in $ L^1(\mcO)$ and
continuous driving signals. The final result will be the unique
existence of a function $X \in L^1(\mcO_T)$ such that the
transformation $Y = e^{\mu}X$ is a generalized weak solution of \eqref
{eqnroughPDEtransformed} (cf. Definition \ref{defgeneralizedweak}
below) as well as its continuity properties (cf. Theorem \ref
{thmctnlimitsoln} below). Defining $X$ to be a solution to \eqref
{eqnroughPDE} is further justified by the construction since $X$ is
obtained as the unique limit of solutions to approximating equations,
independent of the chosen approximating sequence. In order to underline
this fact, to explain the structure of the construction and to point
out the higher regularity of solutions for more regular initial data
and driving signals, we explicitly formulate the intermediate existence
and uniqueness results. We will use the usual notation for (very) weak
solutions as in \cite{DB83}.

\begin{definition}[(Weak and very weak solutions)]\label{defweaksoln}
(i) Let $Y_0 \in L^1(\mcO)$. A~function $Y \in L^1(\mcO_T)$ with
$\Phi(e^{-\mu}Y) \in L^1(\mcO_T)$ is called a very weak solution to
\eqref{eqnroughPDEtransformed} if
%
\begin{equation}
\label{eqnveryweakroughtransformed} -\int_{\mcO_T}
Y_r \partial_r \eta \,d\xi \,dr - \int_\mcO
Y_0 \eta _0 \,d\xi= \int_{\mcO_T} \Phi
\bigl(e^{-\mu_r} Y_r\bigr) \D\bigl(e^{\mu_r}
\eta_r \bigr) \,d\xi \,dr
\end{equation}
for all $\eta\in C^{1,2}(\bar\mcO_T)$ with $\eta= 0$ on $\mcP\mcO_T$.
If in addition $\Phi(e^{-\mu}Y) \in L^1([0,T];\break  W_0^{1,1}(\mcO))$, then
$Y$ is said to be a weak solution to \eqref{eqnroughPDEtransformed}.

(ii) Let $z \in C^{1-\operatorname{var}}([0,T];\R^N)$ and $X_0 \in L^1(\mcO)$.
A function $X \in L^1(\mcO_T)$ such that $t \mapsto ( \int_{\mcO}
B(X_t) \eta_t \,d\xi )$ is continuous and $\Phi(X) \in L^1(\mcO
_T)$ is called a very weak solution to \eqref{eqnroughPDE} if
\begin{eqnarray*}
&&-\int_{\mcO_T} X_r \partial_r \eta \,d
\xi \,dr - \int_\mcO X_0 \eta _0 \,d
\xi
\nonumber
\\[-8pt]
\\[-8pt]
\nonumber
&&\qquad= \int_{\mcO_T} \Phi(X_r) \D\eta_r
\,d\xi \,dr+ \int_0^T \biggl( \int_{\mcO}
B(X_r) \eta_r \,d\xi \biggr) \,dz_r
\end{eqnarray*}
for all $\eta\in C^{1,2}(\bar\mcO_T)$ with $\eta= 0$ on $\mcP\mcO_T$.
If in addition $\Phi(X) \in L^1([0,T];\break  W_0^{1,1}(\mcO))$, then $X$ is
said to be a weak solution to \eqref{eqnroughPDE}.
\end{definition}
Note that in the case of (very) weak solutions to \eqref{eqnroughPDE}
we implicitly assume $z \in C^{1-\operatorname{var}}([0,T];\R^N)$.

A function $Y \in L^1(\mcO_T) \cap C([0,T];H)$ with $\Phi(e^{-\mu} Y)
\in L^1([0,T];H_0^1(\mcO))$ is a weak solution to \eqref
{eqnroughPDEtransformed} if and only if
\[
\frac{dY_t}{dt} = e^{\mu_t}\D\Phi\bigl(e^{-{\mu_t}}Y_t
\bigr)
\]
for a.e. $t \in[0,T]$ as an equation in $H$. Similarly, $X \in
L^1(\mcO_T) \cap C([0,T];H)$ with $\Phi(X) \in L^1([0,T];H_0^1(\mcO))$
is a weak solution to \eqref{eqnroughPDE} if and only if
\[
X_t = X_0 + \int_0^t
\D\Phi(X_r) \,dr + \int_0^t
B(X_r)\,dz_r
\]
for all $t \in[0,T]$ as an equation in $H$. If we replace $H$ by some
weaker space $H^{-k} \supseteq L^1(\mcO)$, then similar equivalences
hold for very weak solutions in $C^w([0,T];L^1(\mcO))$.

For very weak solutions we will prove that equations \eqref
{eqnroughPDE} and \eqref{eqnroughPDEtransformed} are indeed in
one-to-one correspondence under the transformation $Y = e^\mu X$.
%
\begin{theorem}\label{thmtransformationbddvarn}
Let $X_0 \in L^1(\mcO)$, $z \in C^{1-\operatorname{var}}([0,T];\R^N)$ and $X \in
L^1(\mcO_T)$ such that $t \mapsto ( \int_{\mcO} B(X_t) \eta
_t \,d\xi )$ is continuous for all $\eta\in C^{0,2}(\bar\mcO_T)$ with
$\eta= 0$ on $\mcP\mcO_T$. Then $X$ is a very weak solution to
\eqref
{eqnroughPDE} if and only if $Y:= e^{\mu}X$ is a very weak solution
to \eqref{eqnroughPDEtransformed}.
\end{theorem}
As an immediate consequence we obtain that $X$ is a weak solution to
\eqref{eqnroughPDE} if and only if $Y:= e^{\mu}X$ is a weak solution
to \eqref{eqnroughPDEtransformed}. We will prove the following
uniqueness of very weak solutions:
%
\begin{theorem}\label{thmuniquenessveryweaksoln}
Essentially bounded very weak solutions to \eqref{eqnroughPDE} and
\eqref{eqnroughPDEtransformed} are unique.
\end{theorem}
By Theorem \ref{thmtransformationbddvarn}, uniqueness of \eqref
{eqnroughPDE} follows from uniqueness of \eqref
{eqnroughPDEtransformed}. The proof relies on the duality method. We
give a short, informal idea of the proof. For two solutions
$Y^{(1)},Y^{(2)}$ with the same initial condition let $Y =
Y^{(1)}-Y^{(2)}$. Then~$Y$ satisfies
\begin{eqnarray*}
\int_{\mcO_T} Y_r \partial_r\eta \,d\xi
\,dr &=& -\int_{\mcO_T} \bigl( \Phi\bigl(e^{-\mu_r}Y^{(1)}_r
\bigr)-\Phi\bigl(e^{-\mu
_r}Y^{(2)}_r\bigr) \bigr) \D
\bigl(e^{\mu_r} \eta_r \bigr) \,d\xi \,dr
\\
&=& -\int_{\mcO_T} a_r Y_r \D
\bigl(e^{\mu_r} \eta_r \bigr) \,d\xi \,dr
\end{eqnarray*}
for all admissible testfunctions $\eta$, where
\[
a_t:= \cases{ \displaystyle\frac{\Phi(e^{-\mu_t}Y^{(1)}_t)-\Phi(e^{-\mu
_t}Y^{(2)}_t)}{Y^{(1)}_t-Y^{(2)}_t}, & \quad$\mbox{for }
Y^{(1)}_t \ne Y^{(2)}_t,$ \vspace*{2pt}
\cr
0, &\quad $\mbox{otherwise}.$}
\]
We may rewrite this to get
\[
\int_{\mcO_T} Y_r \bigl( \partial_r
\eta_r + a_r \D\bigl(e^{\mu_r}
\eta_r\bigr) \bigr) \,d\xi \,dr = 0
\]
for all admissible $\eta$. The proof is then concluded by choosing
$\eta
$ as a solution to
%
\begin{eqnarray}
\label{eqnvpdefnintro} \partial_r \eta_r +
a_r \D\bigl(e^{\mu_r} \eta_r\bigr) - \t&=& 0\qquad
\mbox{on } \mcO_T,
\nonumber
\\[-8pt]
\\[-8pt]
\nonumber
\eta&=& 0\qquad \mbox{on } \mcP\mcO_T
\end{eqnarray}
for an arbitrary, smooth test-function $\t$. Since solutions to \eqref
{eqnvpdefnintro} are not known to satisfy sufficient regularity be
used as testfunctions, further approximation arguments are required.
These rely on nondegenerate, smooth approximation of $a_r,\mu_r$ and
on an interval splitting method required in order to control the
perturbing factors $e^{\mu_r}$.

As outlined in the \hyperref[sec1]{Introduction} by a nondegenerate approximation of
\eqref{eqnroughPDEtransformed}, we obtain:
%
\begin{theorem}\label{thmexistenceweak} Let $1<m<\infty$. Then:
\begin{longlist}[(ii)]
\item[(i)] Let $Y_0 \in L^\infty(\mcO)$ and $z \in C([0,T];\R^N)$. Then
there exists a unique weak solution $Y \in C([0,T];H) \cap L^\infty
(\mcO
_T)$ to \eqref{eqnroughPDEtransformed} satisfying $\Phi(e^{-\mu} Y)
\in L^2([0,T];H_0^1(\mcO))$. There is a function $U\dvtx[0,T]\times
\mcO
\to\bar\R$ (taking the value $\infty$ at $t=0$) that is piecewise
smooth on $(0,T]$ such that for all $Y_0 \in L^\infty(\mcO)$
\[
Y_t \le U_t\qquad \mbox{a.e. in } \mcO\mbox{ and } \forall t \in[0,T].
\]
($U$ is more explicitly defined in the proof below).
\item[(ii)] Let $z \in C^{1-\operatorname{var}}([0,T];\R^N)$ and $X_0 \in L^\infty
(\mcO
)$. Then there exists a unique weak solution $X \in C([0,T];H)\cap
L^\infty(\mcO_T)$ to \eqref{eqnroughPDE} satisfying $\Phi(X) \in
L^2([0,T];\break  H_0^1(\mcO))$ and $X_t \le U_t$ a.e. in $\mcO$, $\forall t
\in[0,T]$ with a function $U$ as in \textup{(i)}.
\end{longlist}
\end{theorem}
The existence of such an upper bound $U_t$ that is independent of the
initial condition is due to the nonlinearity ($1<m<\infty$) of the
porous medium operator and is well known in the deterministic case (cf.
\cite{V07} and references therein) with $U_t$ being of the form $U_t =
A t^{-{1}/{(m-1)}}(R^2-|\xi|^2)^{{1}/{m}}$.

\begin{remark}\label{rmkfde}
For the case of fast diffusion equations, that is, for $0 < m < 1$ we obtain:
\begin{longlist}[(ii)]
\item[(i)] For $Y_0 \in L^{m+1}(\mcO)$ and $z \in C([0,T];\R^N)$, there
exists a weak solution $Y \in C([0,T];H)$ to \eqref
{eqnroughPDEtransformed} satisfying $\Phi(e^{-\mu} Y) \in
L^2([0,T];H_0^1(\mcO))$. If $Y_0 \in L^\infty(\mcO)$, then
\[
Y_t \le K_t\qquad \mbox{a.e. in } \mcO\mbox{ and } \forall t \in[0,T],
\]
with $K=K(\|Y_0\|_{L^\infty(\mcO)})\dvtx[0,T]\times\mcO\to\R_+$
being a
piecewise smooth function on $[0,T]$. The map $t \mapsto Y_t$ is weakly
continuous in each $L^p(\mcO)$, $p \in[1,\infty)$.
\item[(ii)] Let $z \in C^{1-\operatorname{var}}([0,T];\R^N)$ and $X_0 \in
L^{m+1}(\mcO
)$. Then there exists a weak solution $X \in C([0,T];H)$ to \eqref
{eqnroughPDE} which satisfies $\Phi(X) \in L^2([0,T];\break H_0^1(\mcO))$.
If $X_0 \in L^\infty(\mcO)$ then $X_t \le K_t$ a.e. in $\mcO$,
$\forall t \in[0,T]$ with a function $K$ as in (i). The map $t \mapsto
X_t$ is weakly continuous in each $L^p(\mcO)$, $p \in[1,\infty)$.
\end{longlist}
No uniqueness is obtained for the fast diffusion case.
\end{remark}

So far we can solve \eqref{eqnroughPDEtransformed} for driving
signals being merely continuous while for \eqref{eqnroughPDE} we
require continuous signals of bounded variation. Since we aim to
include rough signals (as they occur, e.g., as sample paths of
fractional Brownian motion) we need to allow rougher signals $z \in
C([0,T];\R^N)$ for \eqref{eqnroughPDE} as well. Such solutions will
be constructed as limits of solutions to smoothed signals $z^{(\ve)}
\in C^{1-\operatorname{var}}([0,T];\R^N)$ with $z^{\ve} \to z$ in $C([0,T];\R^N)$. We
prove that the solutions $X^{(\ve)}$ to \eqref{eqnroughPDE} driven by
these smoothed signals converge to $X:= e^{-\mu}Y$, that is, to a
limit not depending on the chosen approximating sequence. In other
words, $X$ is the limit obtained by any Wong--Zakai approximation of
\eqref{eqnroughPDE}.

\begin{definition}\label{defroughweaksoln}
Let $z \in C([0,T];\R^N)$. We call $X \in C([0,T];H)$ a rough weak
solution to \eqref{eqnroughPDE} if $X(0) = X_0$ and for all
approximations $z^{(\ve)} \in\break   C^{1-\operatorname{var}}([0,T];\R^N)$ of the driving
signal $z$ with $z^{(\ve)} \to z$ in $C([0,T];\R^N)$ and corresponding
weak solutions $X^{(\ve)}$ to \eqref{eqnroughPDE} driven by
$z^{(\ve
)}$, we have
%
\begin{equation}
\label{eqnroughweaksoln} X^{(\ve)}_t \to X_t\qquad
\mbox{in } H\mbox{ and } \forall t \in[0,T].
\end{equation}
\end{definition}

\begin{theorem}\label{thmroughPDE}
Let $X_0 \in L^\infty(\mcO)$ and $z \in C([0,T];\R^N)$. Then there
exists a unique rough weak solution $X$ to \eqref{eqnroughPDE} with
$\Phi(X) \in L^2([0,T];H_0^1(\mcO))$ given by $X=e^{-\mu}Y$, where $Y$
is the corresponding weak solution to \eqref
{eqnroughPDEtransformed}. $X$~satisfies $X_t \le U_t$ a.e. in $\mcO
$ for all $t \in[0,T]$, with $U$ as in Theorem \ref{thmexistenceweak}.
\end{theorem}
The convergence of the approximations $X^{(\ve)} = e^{-\mu^{(\ve
)}}Y^{(\ve)}$ is proven via convergence of $Y^{(\ve)}$. The main point
of the proof of Theorem~\ref{thmroughPDE} is the realization that the
a priori bounds derived in the construction of weak solutions in
Theorem~\ref{thmexistenceweak} rely on the driving signal $z$ only
via its sup-norm and its modulus of continuity. Therefore, the bounds
are uniform for every approximating sequence $z^{(\ve)}$ as in
Definition \ref{defroughweaksoln}. The convergence of $Y^{(\ve)}$ is
then obtained by similar methods as used in the construction of weak solutions.

Since the weak solutions to \eqref{eqnroughPDE} obtained in Theorem
\ref{thmexistenceweak} are also given by $X=e^{-\mu}Y$, the notions
of rough weak solutions and weak solutions to \eqref{eqnroughPDE}
coincide for continuous driving signals of bounded variation and
essentially bounded initial conditions.

\begin{definition}
Let $X_0 \in L^1(\mcO)$ and $z \in C([0,T];\R^N)$. A function $X \in
C^w([0,T];L^1(\mcO))$ is said to be a limit solution to \eqref
{eqnroughPDE} if $X(0) = X_0$ and for all approximations $X^{(\delta)}_0
\in L^\infty(\mcO)$ with $X^{(\delta)}_0 \to X_0$ in $L^1(\mcO)$ and
corresponding rough weak solutions $X^{(\delta)}$ to \eqref{eqnroughPDE},
we have $X^{(\delta)}_t \to X_t$ in $L^1(\mcO)$ uniformly in time.
\end{definition}
These limit solutions play an important role for allowing initial
conditions in $L^1(\mcO)$. In Lemma \ref{lemmaL1-ctnicweaksoln}
below we will establish uniform $L^1(\mcO)$ continuity in the initial
condition for rough weak solutions. This will allow to construct limit
solutions for initial values in $L^1(\mcO)$ by approximation in the
initial condition.

\begin{theorem}\label{thmlimitsoln}
Let $z \in C([0,T];\R^N)$. For each $X_0 \in L^1(\mcO)$ there is a
unique limit solution $X$ satisfying $\Phi(X) \in L^1(\mcO_T)$. For
$X_0^{(i)} \in L^1(\mcO)$, $i=1,2$ the corresponding limit solutions satisfy
\begin{eqnarray*}
&&\sup_{t \in[0,T]} \bigl\|\bigl(X^{(1)}_t-X^{(2)}_t
\bigr)^+\bigr\|_{L^1(\mcO)} + \bigl\|\bigl(\Phi \bigl(X^{(1)}\bigr) - \Phi
\bigl(X^{(2)}\bigr)\bigr)^+\bigr\|_{L^1(\mcO_T)}
\\
&&\qquad\le C \bigl\| \bigl(X^{(1)}_0-X^{(2)}_0
\bigr)^+\bigr\|_{L^1(\mcO)}
\end{eqnarray*}
and
\begin{eqnarray*}
&&\sup_{t \in[0,T]}\bigl\|X^{(1)}_t-X^{(2)}_t
\bigr\|_{L^1(\mcO)} + \bigl\|\Phi \bigl(X^{(1)}\bigr) - \Phi
\bigl(X^{(2)}\bigr)\bigr\|_{L^1(\mcO_T)}
\\
&&\qquad\le C \bigl\| X^{(1)}_0-X^{(2)}_0
\bigr\|_{L^1(\mcO)}.
\end{eqnarray*}
We further have $X_t \le U_t$ a.e. in $\mcO$ for all $t \in[0,T]$,
where $U$ is as in Theorem~\ref{thmexistenceweak}.
\end{theorem}
We present an informal argument justifying the $L^1$ stability
proved in Theorem~\ref{thmlimitsoln} above at least for small times
$T$. Let $\vp\in C^2(\bar\mcO)$ be the unique classical solution to
\begin{eqnarray*}
\D\vp&=& -1\qquad \mbox{in } \mcO,
\\
\vp&=& 1\qquad \mbox{on } \partial\mcO.
\end{eqnarray*}
Using partial integration twice we obtain (informally)
\begin{eqnarray*}
&&\partial_t \int_\mcO\bigl|Y^{(1)}-Y^{(2)}\bigr|
\vp \,d\xi
\\
&&\qquad= \int_\mcO \operatorname{sgn}\bigl(Y^{(1)}-Y^{(2)}
\bigr) \D \bigl(\Phi\bigl(e^{-\mu}Y^{(1)}\bigr)- \Phi
\bigl(e^{-\mu}Y^{(2)}\bigr) \bigr) e^\mu\vp \,d\xi
\\
&&\qquad\le-\int_\mcO \operatorname{sgn}\bigl(\Phi\bigl(e^{-\mu}Y^{(1)}
\bigr)- \Phi\bigl(e^{-\mu}Y^{(2)}\bigr)\bigr)
\\
&&\hspace*{22pt}\qquad\quad{} \times\nabla \bigl(\Phi\bigl(e^{-\mu}Y^{(1)}\bigr)- \Phi
\bigl(e^{-\mu
}Y^{(2)}\bigr) \bigr) \cdot\nabla
\bigl(e^\mu\vp\bigr) \,d\xi
\\
&&\qquad= \int_\mcO\bigl|\Phi\bigl(e^{-\mu}Y^{(1)}
\bigr)- \Phi\bigl(e^{-\mu}Y^{(2)}\bigr)\bigr| \D \bigl(e^\mu
\vp\bigr) \,d\xi.
\end{eqnarray*}
For small times $T$ we note $\D(e^\mu\vp) \le-\frac{1}{2}$, which
yields the result. For large times~$T$ this method is applied by using
an interval splitting technique in order to compensate the growth of
the perturbing factor $e^\mu$.

As a special case we obtain the following comparison principle
%
\begin{corollary}
Let $X_0^{(1)},X_0^{(2)} \in L^1(\mcO)$ with $X_0^{(1)} \le X_0^{(2)}$
a.e. in $\mcO$. Then
\[
X^{(1)}_t \le X^{(2)}_t
\]
for all $t \in[0,T]$, a.e. in $\mcO$. In particular, if $0 \le X_0$,
then $0 \le X$.
\end{corollary}

Let $X$ be a limit solution. By Theorem \ref{thmlimitsoln} there are
rough weak solutions with $X^{(\delta)} \to X$ in $L^\infty
([0,T];L^1(\mcO
))$ and $\Phi(X^{(\delta)}) \to\Phi(X)$ in $L^1(\mcO_T)$. Hence,
there are
weak solutions $Y^{(\delta)} = e^{\mu} X^{(\delta)}$ converging in
$L^\infty
([0,T];L^1(\mcO))$ to $Y:=e^{\mu}X$ and $\Phi(e^{-\mu}Y^{(\delta
)}) \to\Phi
(e^{-\mu}Y)$ in $L^1(\mcO_T)$. Passing to the limit $\delta\to0$ in
\eqref
{eqnveryweakroughtransformed} yields
%
\begin{remark}\label{rmklimitareveryweak}
Let $X_0 \in L^1(\mcO)$ and $X$ be the corresponding limit solution.
Then $Y:=e^{\mu}X$ is a very weak solution of \eqref
{eqnroughPDEtransformed}.
\end{remark}

The limit solution $X$ turns out to be in fact more regular. The proof
proceeds by choosing the approximations used in the\vadjust{\goodbreak} construction of
weak solutions in a way that allows to apply the regularity results
presented in \cite{DB83}. We say that a quantity depends only on the
data if it is a function of $d$, $m$, $T$.

\begin{theorem}\label{thmctnlimitsoln}
Let $z \in C([0,T];\R^N)$, $X_0 \in L^1(\mcO)$ and $X$ be the
corresponding limit solution. Then:
\begin{enumerate}[(iii)]
\item[(i)] $X$ is uniformly continuous on every compact set $K
\subseteq(0,T] \times\mcO$, with modulus of continuity depending only
on the data and $\operatorname{dist}(K,\partial\mcO_T)$.
\item[(ii)] If $X_0 \in L^\infty(\mcO)$ is continuous on a compact set
$K \subseteq\mcO$, then $X$ is uniformly continuous on $[0,T] \times
K'$ for every compact set $K' \subseteq\mathring{K}$, with modulus of
continuity depending only on the data, $\|X_0\|_{L^\infty(\mcO)}$,
$\operatorname{dist}(K,\partial\mcO)$, $\operatorname{dist}(K',\partial K)$ and the modulus of
continuity of $X_0$ on $K$.
\item[(iii)] Assume:
\begin{enumerate}[($\mcO1$)]
\item[($\mcO1$)] There exist $\t^* > 0, R_0 >0$ such that $\forall x_0
\in\partial\mcO$ and every $R \le R_0$
\[
\bigl|\mcO\cap B_R(x_0)\bigr| < \bigl(1-\t^*\bigr)
\bigl|B_R(x_0)\bigr|.
\]
\end{enumerate}
Then for every $\tau> 0$, $X$ is uniformly continuous on $[\tau,T]
\times\bar\mcO$ with modulus of continuity depending only on the data,
$\t^*$ and $\tau$.
\end{enumerate}
\end{theorem}

By dominated convergence we obtain:
%
\begin{corollary}\label{corL^1-timectnlimitsoln}
Let $z \in C([0,T];\R^N)$.
\begin{longlist}[(ii)]
\item[(i)] If $X_0 \in L^1(\mcO)$, then $X \in C([0,T];L^1(\mcO))
\cap
C((0,T];L^p(\mcO))$ for every $p \in[1,\infty)$.
\item[(ii)] If $X_0 \in L^\infty(\mcO)$, then $X \in
C([0,T];L^p(\mcO
))$ for every $p \in[1,\infty)$.
\end{longlist}
\end{corollary}

The continuity obtained in Theorem \ref{thmctnlimitsoln} together
with the $L^\infty$-bounds from Theorem \ref{thmexistenceweak} imply
that the convergence of the various approximating solutions used to
construct limit solutions driven by rough signals in fact holds locally
uniformly. For example we obtain

\begin{corollary}
Let $z \in C([0,T];\R^N)$, $X_0 \in L^1(\mcO)$. Then the convergence
in \eqref{eqnroughweaksoln} holds uniformly on compact sets $K
\subseteq(0,T] \times\mcO$.
\end{corollary}

In Remark \ref{rmklimitareveryweak} we have shown that the limit
solutions $X$ are solutions to \eqref{eqnroughPDE} in the sense that
their transformations $Y:=e^{\mu}X$ are very weak solutions to \eqref
{eqnroughPDEtransformed}. However, since uniqueness of very weak
solutions has only been obtained in the essentially bounded case, this
does not yield a characterization of limit solutions. To overcome this
problem we recall that the limit solutions constructed in Theorem \ref
{thmlimitsoln} enjoy an $L^1-L^\infty$ regularizing property. This
regularization can be used in order to characterize the transformation
$Y:=e^{\mu}X$ of limit solutions $X$ as generalized weak solutions, in
the following sense:

\begin{definition}\label{defgeneralizedweak}
Let $z \in C([0,T];\R^N)$. A map $Y \in C([0,T];L^1(\mcO))$ is said to
be a generalized weak solution to \eqref{eqnroughPDEtransformed} if
$Y$ is an essentially bounded weak solution to \eqref
{eqnroughPDEtransformed} on each interval $[\tau,T]$ with $\tau>
0$; that is, $Y \in L^\infty([\tau,T]\times\mcO)$, $\Phi(e^{-\mu
}Y) \in
L^1([\tau,T];W_0^{1,1}(\mcO))$ and
\[
-\int_{[\tau,T]\times\mcO} Y_r \partial_r \eta\, d
\xi \,dr - \int_{\mcO} Y_\tau\eta_\tau \,d
\xi= -\int_{[\tau,T]\times\mcO} \nabla\Phi \bigl(e^{-\mu
_r}
Y_r\bigr) \cdot\nabla\bigl(e^{\mu_r} \eta_r
\bigr) \,d\xi \,dr
\]
for all $\eta\in C^{1}([\tau,T]\times\bar\mcO)$ with $\eta= 0$ on
$\mcP\mcO_{[\tau,T]}$.

$X \in C([0,T];L^1(\mcO))$ is said to be a generalized weak solution
to \eqref{eqnroughPDE} if $Y=e^{\mu}X$ is a generalized weak solution
to \eqref{eqnroughPDEtransformed}.
\end{definition}

Using the continuity $X \in C([0,T];L^1(\mcO))$ of generalized weak
solutions and Lipschitz continuity of weak solutions in the initial
condition (Theorem~\ref{thmlimitsoln}) we obtain

\begin{proposition}[(Uniqueness of generalized weak solutions)]
Let $X^{(i)}$ be generalized weak solutions with initial conditions
$X^{(i)}_0$, $i=1,2$. Then
\[
\sup_{t \in[0,T]}\bigl\|X^{(1)}_t-X^{(2)}_t
\bigr\|_{L^1(\mcO)} \le C \bigl\| X^{(1)}_0-X^{(2)}_0
\bigr\|_{L^1(\mcO)}.
\]
\end{proposition}

In Theorem \ref{thmlimitsoln} we have obtained that every limit
solution $X$ is essentially bounded on $[\tau,T]\times\mcO$ for all
$\tau> 0$. By uniqueness of limit solutions this implies that $X$ is a
rough weak solution on $[\tau,T]$. Thus $Y=e^{\mu}X$ is a generalized
weak solution.

\begin{theorem}\label{thmcharaclimitsoln}
Let $X_0 \in L^1(\mcO)$, and let $X$ be the corresponding limit
solution to \eqref{eqnroughPDE}. Then $X$ is the unique generalized
weak solution to \eqref{eqnroughPDEtransformed}.
\end{theorem}

\subsection{Stochastic porous medium equation and RDS}

So far we did not require the driving signal to be given by a
stochastic process. We aim to study the long-time behavior of solutions
to PME driven by rough noise. If the rough signal is given by a process
with (strictly) stationary increments this additional structure can be
used to significantly simplify this task. This approach is nicely
captured by the theory of RDS.

For signals given by the paths of a continuous semimartingale
stochastic calculus may be used to give meaning to the integral over
the rough signal occurring in~\eqref{eqnroughPDE}. This allows to
further justify the notion of a rough weak solution which was based on
a Wong--Zakai approximation of the noise (Definition \ref
{defroughweaksoln}).

\begin{theorem}\label{thmtransformationveryweaksemimartingale}
Let $z\dvtx[0,T] \times\O\to\R^N$ be a continuous semimartingale
on a
normal filtered probability space $(\O,\mcF,\mcF_t,\P)$, $X_0 \in
L^0(\O
,\mcF_0;L^1(\mcO))$ and $X(\omega)$ be the corresponding (pathwise) limit
solution to \eqref{eqnroughPDE}. Then
%
\begin{eqnarray}
\label{eqnroughPDEveryweak} \int_\mcO X_t \vp \,d
\xi&= &\int_\mcO X_s \vp \,d\xi + \int
_s^t \int_\mcO
\Phi(X_r) \D\vp \,d\xi \,dr
\nonumber
\\[-8pt]
\\[-8pt]
\nonumber
&&{}+ \int_s^t \biggl( \int_\mcO
B(X_r)\vp \,d\xi \biggr) \circ \,dz_r
\end{eqnarray}
for all $\vp\in C^2_0(\bar\mcO)$ and all $0 \le s \le t \le T$, $\P
$-almost surely.
\end{theorem}
As a part of Theorem \ref{thmtransformationveryweaksemimartingale}
we obtain that $t\mapsto\int_\mcO B(X_t)\vp \,d\xi$ is a continuous
semimartingale with respect to the filtration generated by $z$ for all
$\vp\in C^2_0(\bar\mcO)$. Hence, the stochastic integral in \eqref
{eqnroughPDEveryweak} is well defined.

\begin{remark}
By Theorem \ref{thmroughPDE} we know that for any approximation
$z^{(\ve)} \in C^{1-\operatorname{var}}([0,T];\R^N)$ with $z^{(\ve)} \to z$ in
$C([0,T];\R^N)$ (pathwise) we have $X_t^{(\ve)}(\omega) \to
X_t(\omega)$ in $H$
for all $t \in[0,T]$ and all $\omega\in\O$. Since $X$ is a solution to
\eqref{eqnroughPDEveryweak}, this yields a pathwise Wong--Zakai result.
\end{remark}

\subsubsection{Quasi-continuity of random dynamical systems}
In this section we will first recall basic notions from the theory of
RDS and then develop an existence result for random attractors based on
weakened continuity assumptions for RDS and asymptotic compactness.
This generalized result is needed since the RDS corresponding to \eqref
{eqnroughPDE}, while being continuous on $L^1(\mcO)$, is only
continuous in a weaker sense on $L^p(\mcO)$ for $p \in(1,\infty]$. For
more details on the theory of RDS and random attractors we refer to
\cite{S92,CF94,CDF97,A98}.

In the following let $((\O,\mcF,\P),(\theta_t)_{t \in\R})$ be a
metric dynamical system, that is, $(\Omega,\mathcal{F},\mathbb{P})$ is
a probability space, $(t,\omega) \mapsto\theta_t(\omega)$ is $(\mcB
(\R)
\otimes\mcF,\mcF)$ measurable, $\theta_0 =$ id, $\theta_{t+s} =
\theta
_t \circ\theta_s$ and $\theta_t$ is $\P$-preserving, for all $s,t
\in
\R$.

\begin{definition}
Let $(X,d)$ be a complete and separable metric space. A~random
dynamical system (RDS) over $(\theta_t)_{t \in\R}$ is a measurable map
$\vp\dvtx  \R_+ \times X \times\O\to X$, such that $\vp(0,\omega) =$
id and
\[
\vp(t+s,\omega) = \vp(t,\theta_s \omega) \circ\vp(s,\omega )\qquad
\mbox{(cocycle property)}
\]
for all $t,s \in\R_+$ and $\omega\in\O$. $\vp$ is said to be a
continuous RDS if $x \mapsto\vp(t,\omega)x$ is continuous for all $t
\in\R
_+$ and $\omega\in\O$.
\end{definition}

We now recall the stochastic generalization of notions of absorption,
attraction and $\O$-limit sets.

\begin{definition}
(i) A set-valued map $D\dvtx\O\to2^X$ is called measurable if
$\omega\mapsto D(\omega)$ takes values in the closed subsets of $X$ and
for all $x \in X$ the map $\omega\mapsto d(x,D(\omega))$ is
measurable, where
for nonempty sets $A,B \in2^X$ we set
\[
d(A,B):=\sup_{x \in A} \inf_{y \in B} \,d(x,y)
\]
and $d(x,B) = d(\{x\},B)$. A measurable set-valued map is also called a
(closed) random set.

(ii) A set universe $\mcD$ is a collection of families of
subsets $(D(\omega))_{\omega\in\O}$ of $X$ such that if $D \in\mcD
$ and $\hat
{D}(\omega) \subseteq D(\omega)$ for all $\omega\in\O$, then $\hat
{D} \in\mcD$.
A universe of random sets is a set universe consisting of random closed sets.

(iii) Let $A$, $B$ be random sets. $A$ is said to absorb $B$ if
there exists an absorption time $t_B(\omega)$ such that for all $t \ge
t_B(\omega)$,
\[
\vp(t,\theta_{-t}\omega)B(\theta_{-t}\omega) \subseteq A(
\omega).
\]
$A$ is said to attract $B$ if
\[
d\bigl(\vp(t,\theta_{-t}\omega)B(\theta_{-t}\omega),A(
\omega)\bigr) \mathop {\rightarrow}_{t
\to\infty} 0\qquad \forall\omega\in\O.
\]
\item[(iv)] Let $\mcD$ be a universe of random sets and $D \in\mcD$.
Then $D$ is said to be a \mbox{$\mcD$-}absorbing set for $\vp$ if $D$ absorbs
every set $\td D \in\mcD$. $\mcD$-attracting sets are defined analogously.
\end{definition}
We require absorption and attraction to hold for all $\omega\in\O$ in
order to state our results in their full strength. This is stronger
than usual in the theory of RDS where an exceptional $\P$-zero set is allowed.

\begin{definition}
Let $\mcD$ be a universe of random sets. Then $\vp$ is said to be
$\mcD
$-asymptotically compact in $X$ if the sequence $\vp(t_n,\t
_{-t_n}\omega
)x_n$ has a convergent subsequence in $X$, for all $\omega\in\O$, $t_n
\to\infty$, $x_n \in D(\t_{-t_n}\omega)$ and $D \in\mcD$.
\end{definition}

\begin{definition}
Let $\mcD$ be a universe of random sets. A $\mcD$-random attractor for
an RDS $\vp$ is a compact random set $A \in\mcD$ satisfying:
\begin{longlist}[(ii)]
\item[(i)] $A$ is invariant, that is, $\vp(t,\omega)A(\omega)=
A(\theta_t \omega)$
for all $t > 0$.
\item[(ii)] $A$ is $\mcD$-attracting.
\end{longlist}
\end{definition}

Since we require $A \in\mcD$ the random attractor for an RDS is
uniquely determined.

In \cite{LG08} the assumption of continuity of RDS has been weakened
while preserving sufficient criteria for the existence of random
attractors. This allowed the authors to study RDS on subspaces of their
``original'' state spaces. We prove generalizations of these results
and identify some underlying structures, which will allow to prove the
existence of random attractors for $\vp$ as an RDS on $L^p(\mcO)$ for
all $p\in[1,\infty)$. If condition $(\mcO1)$ is satisfied we will
also obtain the existence of a random attractor with respect to the
$L^\infty$ norm.

\begin{definition}
An RDS $\vp$ on a Banach space $X$ endowed with some topology $\tau$
is said to be quasi-$\tau$-continuous if $\vp(t_n,\omega)x_n \to
^\tau\vp
(t,\omega)x$, whenever $(t_n,x_n) \in\R_+ \times X$ is a sequence such
that $\vp(t_n,\omega)x_n$ is bounded and $(t_n,x_n) \to(t,x)$ for
$n\to
\infty$. Here ``$\to^\tau$'' denotes convergence with respect to
$\tau
$-topology.
\end{definition}
In \cite{LG08} a general result proving quasi-continuity for
restrictions of continuous RDS to subspaces of the state space has been
proven. More precisely:\vspace*{6pt}

\cite{LG08}, Proposition 3.3:
\textit{Let $Y$, $X$ be Banach spaces such that $i\dvtx Y \hookrightarrow X$ and
$i^*\dvtx X^* \hookrightarrow Y^*$ are dense and continuous. If $\vp$
is an
RDS on $X$, $Y$ (resp.) and $\vp$ is (norm-weak) continuous on $X$,
then $\vp$ is quasi-weakly-continuous on $Y$, that is, quasi-$\tau
$-continuous for $\tau$ being the weak topology on $Y$.}\vspace*{6pt}

If $Y$ is a reflexive space, then continuity and density of $i\dvtx Y
\hookrightarrow X$ implies the same for $i^*\dvtx X^* \hookrightarrow Y^*$.
For nonreflexive spaces the situation may be more involved, and, in
general, one may only conclude the existence of the continuous map
$i^*\dvtx X^* \hookrightarrow Y^*$. However, even in the nonreflexive case
$ _{Y^*}\langle\cdot,\cdot\rangle_Y\dvtx i^*(X^*) \times Y \to\R$ defines a
duality mapping,
that is:
\begin{longlist}[(ii)]
\item[(i)] $ _{Y^*}\langle i^*(x^*),y\rangle_Y= 0$ for all $y \in Y$ implies
$i^*(x^*) = 0$,
\item[(ii)] $ _{Y^*}\langle i^*(x^*),y\rangle_Y= 0$ for all $x^* \in
X^*$ implies
$y = 0$.
\end{longlist}
Since $i^*(X^*) \subseteq Y^*$ is a linear subspace and $ _{Y^*}\langle\cdot,\cdot\rangle
_Y\dvtx i^*(X^*) \times Y \to\R$ is a duality mapping, the corresponding
weak topology $\s(Y,\overline{i^*(X^*)})$ on $Y$ is Hausdorff, where
$\overline{i^*(X^*)}$ denotes the closure of $i^*(X^*)$ with respect to
$\|\cdot\|_{Y^*}$. Norm-weak continuity of $\vp$ in $X$ just means
continuity of $(t,x) \mapsto _{X^*}\langle x^*,\vp(t,\omega
)x\rangle_X$ for
all $x^*
\in X^*$, $\omega\in\O$. Hence, norm-weak continuity of $\vp$ in $X$
implies norm-$\s(Y,i^*(X^*))$ continuity on $Y$. On bounded sets $B
\subseteq Y$ we have $\s(Y,i^*(X^*)) \cap B = \s(Y,\overline{i^*(X^*)})
\cap B$. This is the precise idea of quasi-continuity. We obtain:
%
\begin{proposition}\label{propquasi-tau-cont}
Let $X, Y$ be Banach spaces such that $i\dvtx Y \hookrightarrow X$ is
dense and continuous. If $\vp$ is an RDS on $X$, $Y$ and $\vp$ is
(norm-weak) continuous on $X$, then $\vp$ is quasi-$\s(Y,\overline
{i^*(X^*)})$-continuous on $Y$.
\end{proposition}
In the following let $\mcD$ be a universe of random sets and $\k$ be
the Kuratowski measure of noncompactness. We will prove that in the
proof of existence of random attractors the assumption of
omega-limit-compactness can be replaced by\vadjust{\goodbreak} asymptotic compactness. This
indeed weakens the assumptions since every $\mcD$-omega-limit compact
RDS $\vp$, that is, satisfying
\[
\lim_{T \to\infty} \k \biggl( \bigcup_{t\ge T}
\vp(t,\t _{-t}\omega)D(\t _{-t}\omega) \biggr) = 0
\]
for all $\omega\in\O$ and $D \in\mcD$ is $\mcD$-asymptotically compact.

For a topology $\tau$ on a Banach space $X$ and a random set $B$ we
define the $\O$-limit set
\[
\Omega^\tau(B,\omega) = \bigl\{ y \in X| \exists t_n
\to\infty, x_n \in B(\t _{-t_n}\omega),
\vp(t_n,\t_{-t_n}\omega)x_n
\to^\tau y\bigr\}.
\]
$\O$-limit sets with respect to the norm topology are simply denoted by
$\Omega(B,\omega)$. One of the ideas in \cite{LG08} in order to allow
quasi-weak-continuity of $\vp$ is to consider $\O$-limit sets with
respect to the weak topology replacing the usual norm topology. For
asymptotically compact RDS these notions actually coincide:
%
\begin{lemma}\label{lemmaascmpolimit}
Let $\vp$ be a $\mcD$-asymptotically compact RDS on the Banach space
$X$ endowed with a Hausdorff topology $\tau$ that is weaker than the
norm topology. Then
\[
\Omega(B,\omega) = \Omega^\tau(B,\omega)\qquad \forall B \in\mcD.
\]
\end{lemma}

In the proof of existence of random attractors we can replace $\mcD
$-omega-limit-compactness by $\mcD$-asymptotic compactness due to the
following observation
%
\begin{lemma}\label{lemmaascmpolimit2}
Let $\vp$ be a $\mcD$-asymptotically compact, quasi-$\tau$-continuous
RDS on the Banach space $X$ endowed with a Hausdorff topology $\tau$
that is weaker than the norm topology. Further assume that there is a
bounded $\mcD$-attracting set~$F$. Then $\Omega(B,\omega)$ is a nonempty,
compact, invariant set for each $B \in\mcD$, $B \ne\varnothing$,
$\omega
\in\O$.
\end{lemma}
If we work with the weaker notion of absorption occurring only $\P
$-a.s., then invariance in Lemma \ref{lemmaascmpolimit2} is
satisfied only crudely. That is $\vp(t,\omega)\O(B,\omega)=\O(B,\t
_t\omega)$ on a
$\P$-zero that may depend on $t$. In the proof of the existence of
random attractors this obstacle can be resolved by a ``perfection''
result proving that there is an indistinguishable, perfectly invariant
modification of $\O(B,\omega)$.

With these preparations it is easy to see that the proof of \cite{LG08}, Theorem 4.1, can be modified so that only quasi-$\tau
$-continuity and asymptotic compactness with respect to the universe of
all bounded deterministic sets has to be assumed.

In our case the universe of absorbed sets will be much larger than just
deterministic bounded sets. This allows us to drop the assumption of
ergodicity of the underlying metric dynamical system. In conclusion we
obtain the following:
%
\begin{theorem}\label{thmexra}
Let $\vp$ be a quasi-$\tau$-continuous RDS on a Banach space $X$,
where $\tau$ is a Hausdorff topology that is weaker than the norm
topology. Then $\vp$ has a $\mcD$-random attractor if and only if:
\begin{longlist}[(ii)]
\item[(i)] $\vp$ has a bounded $\mcD$-attracting random set $F \in
\mcD$.
\item[(ii)] $\vp$ is $\mcD$-asymptotically compact in $X$.
\end{longlist}
\end{theorem}

\subsection{\texorpdfstring{RDS and random attractors for \protect\eqref{eqnroughPDE}}
{RDS and random attractors for (1.2)}}

Let $(\Omega,\mathcal{F},\mathcal{F}_t,\mathbb{P})$ be a filtered
probability space, $(z_t)_{t \in\R}$ be an $\R^N$-valued adapted
stochastic process and $((\O,\mcF,\P),(\theta_t)_{t \in\R})$ be a
metric dynamical system. We assume
\begin{longlist}[$(S1)$]
\item[$(\mathrm{S1})$] (Strictly stationary increments).\footnote{This property
is also called ``perfect helix property'' \cite{AS95}.} For all $t,s
\in\R$, $\omega\in\O$
\[
z_t(\omega)-z_s(\omega) = z_{t-s}(
\t_s \omega).
\]
We assume $z_0 = 0$ for notational convenience only.
\item[$(\mathrm{S2})$] (Regularity). $z_t$ has continuous paths.
\end{longlist}

Adaptedness and $(\mathrm{S2})$ imply joint measurability of $z$, that is,
$z\dvtx\R\times\O\to\R^N$ is $(\mcB(\R) \otimes\mcF,\mcB(\R^N))$
measurable. Note
\[
\label{eqnmustrongstat} \mu_t(\omega)-\mu_s(\omega) =
\sum_{k=1}^N f_k
\bigl(z^k_t(\omega )-z^k_s(
\omega)\bigr) = \mu _{t-s}(\t_s \omega),
\]
and recall that $f_k$ are functions depending on the space variable.

By \cite{GLR11}, Lemma 3.1, for each $\R^N$ valued process $\td z_t$
with $\td z_0 = 0$ a.s., stationary increments and a.s. continuous
paths there exists a metric dynamical system $((\O,\mcF,\P),(\theta
_t)_{t \in\R})$ and a version $z_t$ of $\td z_t$ on $((\O,\mcF,\P
),(\theta_t)_{t \in\R})$ such that $z_t$ satisfies $(\mathrm{S1})$, $(\mathrm{S2})$. In
particular, applications include fractional Brownian motion with
arbitrary Hurst parameter.

Using the pathwise results obtained in Section~\ref{ssecRPME}, we
define the RDS $\vp$ on $X:= L^1(\mcO)$ associated to \eqref
{eqnroughPDE}. For $t \ge s$, $\omega\in\O$ and $x \in L^1(\mcO)$ let
$X(t,s;\omega)x$ denote the unique limit solution to \eqref{eqnroughPDE}
on $[s,\infty)$ with $X_s = x$ and driving signal $z = z(\omega)$.
%
\begin{definition}\label{eqndefvpflow}
For $t \ge s$, $\omega\in\O$ and $x \in L^1(\mcO)$ define
\[
\vp(t-s,\t_s \omega)x:= X(t,s;\omega)x.
\]
\end{definition}

\begin{theorem}\label{thmgenerationRDS}
The map $\vp$ from Definition \ref{eqndefvpflow} is a continuous
RDS on $X=L^1(\mcO)$ and thus a quasi-weakly-continuous RDS on each
$L^p(\mcO)$, $p \in[1,\infty)$. In addition, $\vp$ is a
quasi-weakly$^*$-continuous RDS on $L^\infty(\mcO)$. $\vp$ satisfies
comparison, that is, for $x_1,x_2 \in X$ with $x_1 \le x_2$ a.e. in
$\mcO$
\[
\vp(t,\omega)x_1 \le\vp(t,\omega)x_2\qquad\mbox{a.e. in }
\mcO.
\]
Moreover, $\vp$ satisfies $\vp(t,\omega)0 = 0$ and:
\begin{longlist}[(iii)]
\item[(i)] $x \mapsto\vp(t,\omega)x$ is Lipschitz continuous on $X$,
locally uniformly in $t$.
\item[(ii)] $t \mapsto\vp(t,\omega)x$ is continuous in $X$.
\item[(iii)] $\vp(t,\omega)x \le U_t(\omega)$ a.e. in $\mcO$ for
all $t \ge
0$, $\omega\in\O$, with $U$ as in Theorem~\ref{thmexistenceweak}.
\item[(iv)] $\vp$ satisfies the same regularity properties as for the
pathwise solutions obtained in Theorem \ref{thmctnlimitsoln}.
\end{longlist}
\end{theorem}
For the general theory of order preserving, continuous RDS we refer to
\cite{C02} and the references therein.

Let $\mcD$ be the universe of all random closed sets in $X$. Using the
uniform $L^\infty$ bound obtained in Theorem \ref{thmgenerationRDS}
we obtain the existence of a $\mcD$-absorbing set $F$ which is bounded
even in $L^\infty(\mcO)$. In fact, the absorption time $t_D(\omega)$
can be
chosen independently of $\omega$ and $D$; cf. Proposition \ref
{propbddabs} below.

If the domain $\mcO$ satisfies condition $(\mcO1)$ by combining the
uniform $L^\infty(\mcO)$ estimate and Theorem \ref{thmgenerationRDS}(iii), we will conclude that the set $\vp(\delta,\omega)F(\omega)$
with $\delta> 0$ is
compact in $C^0(\bar\mcO)$ and $\mcD$-absorbing in $\mcD$. By Theorem
\ref{thmexra} this implies the existence of a $\mcD$-random
attractor. If the domain $\mcO$ does not necessarily satisfy condition
$(\mcO1)$ we only get inner continuity, that is, equicontinuity of
$\vp
(\delta,\omega)F(\omega)$ on each compact set $K \subseteq\mcO$.
In this case we
cannot conclude the existence of a compact $\mcD$-absorbing set, but we
can still prove $\mcD$-asymptotic compactness for $\vp$. By Theorem
\ref
{thmexra} we arrive at the following:

\begin{theorem}\label{thmexistenceRA}
Let $\mcD$ be the universe of all random closed sets in $L^1(\mcO)$.
The RDS $\vp$ has a $\mcD$-random attractor $A$ [as an RDS on
$L^1(\mcO
)$]. $A$~is compact in each $L^p(\mcO)$ and attracts all sets in $\mcD$
in $L^p$-norm, $p \in[1,\infty)$.

Moreover, $A(\omega)$ is a bounded set in $L^\infty(\mcO)$ and the
functions in $A(\omega)$ are equicontinuous on every compact set $K
\subseteq\mcO$.

If $(\mcO1)$ is satisfied, then $A(\omega)$ is a compact set in
$C^0(\bar
\mcO)$ and attracts all sets in $\mcD$ in $L^\infty$-norm.
\end{theorem}

%
%

\section{Porous medium equation driven by rough signals}\label{secroughcase}
%
%

\subsection{Transformation for signals of bounded variations}
In this section we prove Theorem \ref{thmtransformationbddvarn}. Let
$z \in C^{1-\operatorname{var}}([0,T];\R^N)$, $\eta\in C^{1,2}(\bar\mcO_T)$ with
$\eta= 0$ on $\mcP\mcO_T$, and let $X$ be a very weak solution to
\eqref{eqnroughPDE}. We prove that $Y:= e^{\mu}X$ is a very weak
solution to \eqref{eqnroughPDEtransformed}. Let $z^{\ve} \in
C^1([0,T];\R^N)$ such that $z^{(\ve)} \to z$ in $C([0,T];\R^N)$ with
uniformly bounded variation, that is, $\sup_{\ve> 0} \|
z^{(\ve
)}\|_{C^{1-\operatorname{var}}} < \infty$. Define $\mu^{\ve}$ as in \eqref
{eqnmudef}. Then
\[
- \int_{\mcO_T} Y_r \partial_r
\eta_r \,d\xi \,dr = - \lim_{\ve\to0} \int
_{\mcO_T} X_r e^{\mu^{(\ve)}_r} \partial_r
\eta_r \,d\xi \,dr
\]
and
\begin{eqnarray*}
&& - \int_{\mcO_T} X_r e^{\mu^{(\ve)}_r}
\partial_r \eta_r \,d\xi \,dr
\\
&&\quad= - \int_{\mcO_T} X_r \partial_r
\bigl( e^{\mu^{(\ve)}_r} \eta_r \bigr) \,d\xi \,dr + \int
_{\mcO_T} X_r \eta_r
\partial_r e^{\mu
^{(\ve
)}_r} \,d\xi \,dr
\\
&&\quad= \int_{\mcO} X_0 e^{\mu^{(\ve)}_0}
\eta_0 \,d\xi+ \int_{\mcO
_T} \Phi (X_r)
\D \bigl( e^{\mu^{(\ve)}_r} \eta_r \bigr) \,d\xi \,dr
\\
&&\qquad{} + \int_0^T \biggl( \int
_{\mcO} B(X_r) \bigl( e^{\mu
^{(\ve
)}_r}
\eta_r \bigr) \,d\xi \biggr) \,dz_r - \int
_{\mcO_T} X_r \eta_r e^{\mu^{(\ve)}_r}
\partial_r \mu^{(\ve)} \,d\xi \,dr
\\
&&\quad= \int_{\mcO} X_0 e^{\mu^{(\ve)}_0}
\eta_0 \,d\xi+ \int_{\mcO
_T} \Phi (X_r)
\D \bigl( e^{\mu^{(\ve)}_r} \eta_r \bigr) \,d\xi \,dr
\\
&&\qquad{} + \int_0^T \biggl( \int
_{\mcO} B(X_r) \bigl( e^{\mu
^{(\ve
)}_r}
\eta_r \bigr) \,d\xi \biggr)\,dz_r - \int
_0^T \biggl( \int_{\mcO}
B(X_r) \bigl( \eta_r e^{\mu^{(\ve)}_r} \bigr) \,d\xi
\biggr) \,dz_r^{(\ve)}.
\end{eqnarray*}
By continuity of the Riemann--Stieltjes integral with respect to the
convergence $z^{(\ve)} \to z$ specified above and uniform convergence
of the integrands (cf. \cite{FV10}, Proposition 2.7), we can take the
limit $\ve\to0$ to obtain the assertion. The other implication
follows by similar arguments.

%
%

%
\subsection{Uniqueness of essentially bounded very weak solutions}%
We prove Theorem \ref{thmuniquenessveryweaksoln}. The proof uses
ideas first developed in \cite{BCM84} combined with interval splitting
techniques that have also been used in \cite{BR11b}. Let $Y^{(1)}, Y^{(2)}$ be two essentially bounded very weak solutions to \eqref
{eqnroughPDEtransformed} with the same initial condition $Y_0 \in
L^1(\mcO)$, and let $Y = Y^{(1)}-Y^{(2)}$. Then
\begin{eqnarray*}
\int_{\mcO_T} Y_r \partial_r\eta \,d\xi
\,dr &=& -\int_{\mcO_T} \bigl( \Phi\bigl(e^{-\mu_r}Y^{(1)}_r
\bigr)-\Phi\bigl(e^{-\mu
_r}Y^{(2)}_r\bigr) \bigr) \D
\bigl(e^{\mu_r} \eta_r \bigr) \,d\xi \,dr
\\
&=& -\int_{\mcO_T} a_r Y_r \D
\bigl(e^{\mu_r} \eta_r \bigr) \,d\xi \,dr
\end{eqnarray*}
for all $\eta\in C^{1,2}(\bar\mcO_T)$ with $\eta= 0$ on $\mcP\mcO
_T$, where
\[
a_t:= \cases{ \displaystyle\frac{\Phi(e^{-\mu_t}Y^{(1)}_t)-\Phi(e^{-\mu
_t}Y^{(2)}_t)}{Y^{(1)}_t-Y^{(2)}_t}, & \quad$\mbox{for }
Y^{(1)}_t \ne Y^{(2)}_t,$ \vspace*{2pt}
\cr
0, & \quad$\mbox{otherwise}.$}
\]
Let $z^{\ve} \in C^\infty([0,T];\R^N)$ with $z^{(\ve)} \to z$ in
$C([0,T];\R^N)$ such that for $\mu^{(\ve)}$ as in~\eqref
{eqnmudef} we
have $\sup_{t \in[0,T]}\|e^{\mu^{(\ve)}_t}-e^{\mu_t}\|_{C^2(\mcO
)} \le
\ve^2$. By equicontinuity of $z^{(\ve)}$ we can choose a partition $0 =
\tau_0 < \cdots < \tau_N = T$ such that
%
\begin{eqnarray}
\label{eqnchoicetau2} \delta&:=& \bigl\|e^{\mu} \bigl( 2\bigl|\nabla\bigl(
\mu^{(\ve)}-\mu_{\tau_i}\bigr)\bigr|^4 + 2\bigl|\D \bigl(
\mu^{(\ve)}-\mu_{\tau_i}\bigr)\bigr|^2 \nonumber\\
&&\hspace*{106pt}{}+ \bigl|\nabla\bigl(\mu^{(\ve)}-\mu_{\tau_i}\bigr)\bigr|^2
\bigr) \bigr\|_{L^\infty
([\tau
_i,\tau_{i+1}] \times\mcO)} \\
&<& \frac{1}{16C}\nonumber
\end{eqnarray}
for all $i = 0,\ldots,N-1$, $\ve> 0$, where $C$ is a constant that will
be specified below. Let $\g= \max_i\{ |\tau_{i+1}-\tau_i| \}$.
We prove $Y = 0$ a.e. via induction over \mbox{$i=0,\ldots,N-1$}. Thus assume $Y
= 0$ on ${[0,\tau_i] \times\mcO}$ almost everywhere. We can modify
$\tau_i$ so that~\eqref{eqnchoicetau2} is preserved and $Y(\tau
_i) =
0$ a.e. in $\mcO$. Define $\mcO_i:= [\tau_i,\tau_{i+1}] \times
\mcO$. Then
\[
\int_{\mcO_i} Y_r \bigl( \partial_r
\eta_r + a_r \D\bigl(e^{\mu_r}
\eta_r\bigr) \bigr) \,d\xi \,dr = 0
\]
for all $\eta\in C^{1,2}([\tau_i,\tau_{i+1}] \times\bar\mcO)$ with
$\eta= 0$ on $\mcP\mcO_{[\tau_{i},\tau_{i+1}]}$.

For $Y_t^{(1)} \ne Y_t^{(2)}$ we have $a_t = e^{-\mu_t} \dot\Phi
(\zeta
_t)$ with $\z_t \in[e^{-\mu_t} Y_t^{(1)},e^{-\mu_t} Y_t^{(2)}]$ and
thus $\|a\|_{L^\infty(\mcO_T)} < \infty$ by essential boundedness of
$Y^{(i)}$.
We consider a nondegenerate, smooth approximation of $a$. Set $\hat
a_\ve:= a \vee\ve$ and let $a_{\ve,\delta}$ be a smooth
approximation of
$\hat a_\ve$ such that $a_{\ve,\delta} \ge\ve$ and $\int_{\mcO
_T} |Y|^2
(\hat a_\ve- a_{\ve,\delta})^2  \,dx \,dr \le\delta$. Then choose
$a_{\ve} =
a_{\ve,\ve^2}$.

Let $\eta= e^{-\mu_{\tau_i}} \vp$ with $\vp$ being the classical
solution to
%
\begin{eqnarray}
\label{eqnvpdefn} \partial_t\vp+ a_\ve e^{\mu_{\tau_i}}
\D\bigl(e^{\mu^{(\ve)}-\mu
_{\tau_i}} \vp\bigr) - \t&=& 0\qquad \mbox{on } \mcO_T,
\nonumber
\\[-8pt]
\\[-8pt]
\nonumber
\vp&= &0\qquad \mbox{on } \mcP\mcO_{[\tau_i,\tau_{i+1}]},
\end{eqnarray}
where $\t$ is an arbitrary smooth testfunction. Time inversion
transforms \eqref{eqnvpdefn} into a uniformly parabolic linear
equation with smooth coefficients. Thus, unique existence of classical
solutions to \eqref{eqnvpdefn} follows from \cite{LSU67}, Theorem 6.2, page
457. Then
%
\begin{eqnarray}
\label{eqnuniqueness1}\qquad  0 &=& \int_{\mcO_i} Y_r
\bigl( \partial_r \eta+ a_r \D\bigl(e^{\mu_r}
\eta _r\bigr) \bigr) \,d\xi \,dr
\nonumber
\\
&=& \int_{\mcO_i} Y_r \bigl( \partial_r
\eta+ a_{\ve,r} \D\bigl(e^{\mu
^{(\ve
)}_r} \eta\bigr) \bigr) \,d\xi \,dr + \int
_{\mcO_i} Y_r (a_r - a_{\ve,r})
\D \bigl(e^{\mu_r^{(\ve)}} \eta_r\bigr) \,d\xi \,dr
\nonumber
\\
&&{}+ \int_{\mcO_i} Y_r a_{r} \D\bigl(
\bigl(e^{\mu_r}-e^{\mu^{(\ve)}_r}\bigr) \eta_r\bigr) \,d\xi \,dr
\\
&= &\int_{\mcO_i} e^{-\mu_{\tau_i}} Y_r
\t_r \,d\xi \,dr + \int_{\mcO_i} Y_r
(a_r - a_{\ve,r}) \D\bigl(e^{\mu_r^{(\ve)}-\mu_{\tau_i}}
\vp_r\bigr) \,d\xi \,dr\nonumber
\\
&&{} + \int_{\mcO_i} Y_r a_{r} \D\bigl(
\bigl(e^{\mu_r}-e^{\mu^{(\ve)}_r}\bigr) e^{-\mu_{\tau_i}} \vp_r
\bigr) \,d\xi \,dr.
\nonumber
\end{eqnarray}
We need to prove that the last two terms vanish for $\ve\to0$. For
this we first derive a bound for $\int_{\mcO_i} a_{\ve,r} |\D
(e^{\mu
_r^{(\ve)}-\mu_{\tau_i}} \vp_r)|^2  \,d\xi \,dr$. Let $\z\in C^\infty
(\R
)$ with $\z(\tau_i)=0$, $\z\le1$ on $[0,T]$ and $\dot\z\ge c > 0$
for some $c \le\frac{1}{4 \g}$. Multiplying \eqref{eqnvpdefn} by
$\z
\D\vp$ and integrating yields
\begin{eqnarray*}
&&\int_{\mcO_i} (\partial_r\vp_r )
\z_r \D\vp_r \,d\xi \,dr
\\
&&\qquad = \int_{\mcO_i}
\bigl( -a_{\ve,r} e^{\mu_{\tau_i}} \D\bigl(e^{\mu_r^{(\ve
)}-\mu
_{\tau_i}} \vp\bigr)
\z_r \D\vp_r + \t_r \z_r \D
\vp_r \bigr) \,d\xi \,dr.
\end{eqnarray*}
Note that
\begin{eqnarray*}
\D\vp&=& \D \bigl( e^{-(\mu^{(\ve)}-\mu_{\tau_i})} e^{\mu^{(\ve
)}-\mu
_{\tau_i}} \vp \bigr)
\\
&=& \vp\bigl(-\bigl|\nabla\bigl(\mu^{(\ve)}-\mu_{\tau_i}
\bigr)\bigr|^2 - \D\bigl(\mu^{(\ve
)}-\mu _{\tau_i}\bigr)
\bigr) -2 \nabla\bigl(\mu^{(\ve)}-\mu_{\tau_i}\bigr) \nabla\vp
\\
&&{}+ e^{-(\mu^{(\ve)}-\mu_{\tau_i})} \D e^{\mu^{(\ve
)}-\mu
_{\tau_i}} \vp.
\end{eqnarray*}
Hence
\begin{eqnarray*}
&&\frac{1}{2} \int_{\mcO_i} |\nabla\vp_r|^2
\dot\z_r \,d\xi \,dr + \int_{\mcO_i} a_{\ve,r}
e^{2\mu_{\tau_i}-\mu_r^{(\ve)}}\bigl |\D\bigl(e^{\mu
_r^{(\ve
)}-\mu_{\tau_i}} \vp_r
\bigr)\bigr|^2 \z_r \,d\xi \,dr
\\
&&\qquad= \int_{\mcO_i} a_{\ve,r} \z_r
e^{\mu_{\tau_i}} \bigl|\vp_r \D \bigl(e^{\mu
_r^{(\ve)}-\mu_{\tau_i}} \vp_r
\bigr) \bigr|\\
&&\hspace*{48pt}{}\times \bigl( \bigl|\nabla\bigl(\mu_r^{(\ve
)}-\mu_{\tau
_i}
\bigr)\bigr|^2 + \bigl|\D\bigl(\mu_r^{(\ve)}-
\mu_{\tau_i}\bigr)\bigr| \bigr) \,d\xi \,dr
\\
&&\qquad\quad{} +\int_{\mcO_i} 2 \bigl( a_{\ve,r} \z_r
e^{\mu_{\tau
_i}} \bigl|\D \bigl(e^{\mu_r^{(\ve)}-\mu_{\tau_i}} \vp_r\bigr)\bigr| \bigl|\nabla
\bigl(\mu_r^{(\ve
)}-\mu_{\tau
_i}\bigr)\bigr| |\nabla
\vp_r| \\
&&\hspace*{230pt}{}+ \t_r \z_r \D\vp_r
\bigr) \,d\xi \,dr.
\end{eqnarray*}
The first term on the right-hand side is bounded by
\begin{eqnarray*}
&&\int_{\mcO_i} a_{\ve,r} \z_r
e^{\mu_{\tau_i}} \bigl|\vp_r \D\bigl(e^{\mu
_r^{(\ve
)}-\mu_{\tau_i}} \vp_r
\bigr) \bigr| \bigl(\bigl |\nabla\bigl(\mu_r^{(\ve)}-\mu_{\tau_i}
\bigr)\bigr|^2 +\bigl |\D\bigl(\mu_r^{(\ve)}-
\mu_{\tau_i}\bigr)\bigr| \bigr) \,d\xi \,dr
\\
&&\qquad\le\int_{\mcO_T} \frac{1}{4} a_{\ve,r}
\z_r e^{2\mu_{\tau
_i}-\mu
_r^{(\ve)}} \bigl|\D\bigl(e^{\mu_r^{(\ve)}-\mu_{\tau_i}} \vp_r
\bigr)\bigr|^2 \,d\xi \,dr + C \delta \int_{\mcO_T} \dot
\z_r |\nabla\vp_r|^2 \,d\xi \,dr
\end{eqnarray*}
and the second by
\begin{eqnarray*}
&&\int_{\mcO_i} \bigl(2 a_{\ve,r} \z_r
e^{\mu_{\tau_i}} \bigl|\D\bigl(e^{\mu
_r^{(\ve
)}-\mu_{\tau_i}} \vp_r\bigr)\bigr| \bigl|\nabla
\bigl(\mu_r^{(\ve)}-\mu_{\tau_i}\bigr)\bigr| |\nabla
\vp_r| + \t_r \z_r \D\vp_r
\bigr) \,d\xi \,dr
\\
&&\qquad\le\int_{\mcO_i} \frac{1}{4} a_{\ve,r}
\z_r e^{2\mu_{\tau
_i}-\mu
_r^{(\ve)}} \bigl|\D\bigl(e^{\mu_r^{(\ve)}-\mu_{\tau_i}} \vp_r
\bigr)\bigr|^2 \,d\xi \,dr
\\
&&\hspace*{16pt}\qquad\quad{}+ \biggl(C \delta+ \frac{1}{8}\biggr) \int_{\mcO_i} \dot
\z_r |\nabla\vp _r|^2 \,d\xi \,dr + C \int
_{\mcO_i} |\nabla\t_r|^2 \,d\xi \,dr.
\end{eqnarray*}
Using this we obtain
\begin{eqnarray*}
&&\frac{1}{2} \int_{\mcO_i} |\nabla\vp_r|^2
\dot\z_r \,d\xi \,dr + \int_{\mcO_i} a_{\ve,r}
e^{2\mu_{\tau_i}-\mu_r^{(\ve)}}\bigl |\D\bigl(e^{\mu
_r^{(\ve
)}-\mu_{\tau_i}} \vp_r
\bigr)\bigr|^2 \z_r \,d\xi \,dr
\\
&&\qquad\le\int_{\mcO_i} \frac{1}{2} a_{\ve,r}
\z_r e^{2\mu_{\tau
_i}-\mu
_r^{(\ve)}} \bigl|\D\bigl(e^{\mu_r^{(\ve)}-\mu_{\tau_i}} \vp_r
\bigr)\bigr|^2 \,d\xi \,dr
\\
&&\hspace*{16pt}\qquad\quad{} + \biggl(2C \delta+ \frac{1}{8}\biggr) \int_{\mcO_i}
\dot\z_r |\nabla\vp _r|^2 \,d\xi \,dr + C \int
_{\mcO_i} |\nabla\t_r|^2 \,d\xi \,dr,
\end{eqnarray*}
where $C = C(\|Y^{(i)}\|_{L^\infty([0,T]\times\mcO)},T,\|e^\mu\|
_{L^\infty(\mcO_T)})$ is a generic constant. Since $C$ is independent
of the choice of $\z$, using Fatou's lemma and \eqref{eqnchoicetau2}
we obtain
%
\begin{eqnarray}
\label{eqnDvpbound} \int_{\mcO_i} a_{\ve,r}
e^{2\mu_{\tau_i}-\mu_r^{(\ve)}} \bigl|\D \bigl(e^{\mu
_r^{(\ve)}-\mu_{\tau_i}} \vp_r
\bigr)\bigr|^2 \,d\xi \,dr \le C \int_{\mcO_i} |\nabla
\t_r|^2 \,d\xi \,dr.
\end{eqnarray}
By the choice $a_\ve$ we have
\begin{eqnarray*}
&&\int_{\mcO_i} |Y_r|^2
\frac{(a_r - a_{\ve,r})^2}{a_{\ve,r}} \,d\xi \,dr
\\
&&\qquad\le\frac{1}{\ve} \biggl( \int_{\mcO_i}
2|Y_r|^2 (a_r - \hat a_{\ve
,r})^2
\,d\xi \,dr + \int_{\mcO_i} 2|Y_r|^2 (\hat
a_{\ve,r} - a_{\ve
,r})^2 \,d\xi \,dr \biggr)
\\
&&\qquad\le4\ve\int_{\mcO_i} |Y_r|^2 \,d\xi \,dr.
\end{eqnarray*}
For the second term in \eqref{eqnuniqueness1} we obtain
\begin{eqnarray*}
&&\int_{\mcO_i} Y_r (a_r -
a_{\ve,r}) \D\bigl(e^{\mu_r^{(\ve)}-\mu
_{\tau_i}} \vp_r\bigr) \,d\xi \,dr
\\
&&\qquad \le \biggl( \int_{\mcO_i} |a_{\ve,r}| \bigl|\D
\bigl(e^{\mu_r^{(\ve)}-\mu
_{\tau
_i}} \vp_r\bigr)\bigr|^2 \,d\xi \,dr
\biggr)^{{1}/ {2}} \biggl( \int_{\mcO_i}
|Y_r|^2 \frac{(a_r - a_{\ve,r})^2}{a_{\ve,r}} \,d\xi \,dr \biggr)^
{{1} /{2}}
\\
&&\qquad\le C \sqrt{\ve} \|\nabla\t\|_{L^2(\mcO_i)} \|Y\|_{L^2(\mcO_i)} \to0
\end{eqnarray*}
for $\ve\to0$. For the third term in \eqref{eqnuniqueness1} we use
\eqref{eqnDvpbound} and $a_\ve\ge\ve$ to get
\[
\|\vp\|_{H^2(\mcO)} \le\frac{C \|\nabla\theta\|_{L^2(\mcO
_i)}}{\ve}.
\]
Hence,
\begin{eqnarray*}
\int_{\mcO_i} Y_r a_{r} \D\bigl(
\bigl(e^{\mu_r}-e^{\mu^{(\ve)}_r}\bigr) e^{-\mu
_{\tau
_i}} \vp_r
\bigr) \,d\xi \,dr
 &\le& C\|\nabla\theta\|_{L^2(\mcO_i)}\frac{\|e^{\mu_r}-e^{\mu
^{(\ve
)}_r}\|_{C^2(\mcO)} }{\ve}
\\
& \le&\ve C\|\nabla\theta\|_{L^2(\mcO_i)} \to0
\end{eqnarray*}
for $\ve\to0$. Taking $\ve\to0$ in \eqref{eqnuniqueness1} yields
\[
0 = \int_{\mcO_i} e^{-\mu_{\tau_i}} Y_r
\t_r \,d\xi \,dr
\]
for any smooth testfunction $\t$. Thus $Y = 0$ in $\mcO_i = [\tau
_i,\tau
_{i+1}] \times\mcO$ almost everywhere. Induction now completes the proof.

\begin{remark}
The method to prove uniqueness used above fails for fast diffusion
equations, since the difference quotient
\[
a_t:= \cases{ \displaystyle\frac{\Phi(e^{-\mu_t}Y^{(1)}_t)-\Phi(e^{-\mu
_t}Y^{(2)}_t)}{Y^{(1)}_t-Y^{(2)}_t}, & \quad$\mbox{for }
Y^{(1)}_t \ne Y^{(2)}_t,$ \vspace*{2pt}
\cr
0, & \quad$\mbox{otherwise}$}
\]
it not known to remain bounded.
\end{remark}

%
%

\subsection{Weak solutions and uniform bounds}

We will now prove Theorem \ref{thmexistenceweak}. In order to
construct weak solutions to \eqref{eqnroughPDEtransformed} several
steps are needed. First we will consider approximating equations, where
the degenerate nonlinearity $\Phi$ is replaced by nondegenerate
functions $\Phi^{(\delta)}$ and the driving signals $z$ are
approximated by
smooth signals $z^{(\delta)}$ (Section~\ref{ssecnondegapprox}).
Existence of classical solutions to these equations follows from
well-known existence results; cf., for example, \cite{LSU67}. Then we
will prove uniform $L^\infty$ bounds for these approximating solutions
(Section~\ref{ssecL-infty-bound}) which will be used in Section~\ref{ssecweaksoln} to finally construct weak solutions to \eqref
{eqnroughPDEtransformed} by monotonicity methods.

\subsubsection{Nondegenerate, smooth approximation and classical
solutions}\label{ssecnondegapprox}
For $\delta> 0$ we choose an approximating function $\Phi^{(\delta)}
\in
C^\infty(\R)$ such that:
\begin{longlist}[(iii)]\label{eqnapproxprop}
\item[(i)] $\Phi^{(\delta)}(0) = 0$ and $\Phi^{(\delta)}$ is
anti-symmetric in $0$;
\item[(ii)] $\Phi^{(\delta)}(r) = \Phi(r)$, for all $\delta\le|r|
\le\frac
{1}{\delta}$;
\item[(iii)] for all $r \in\R$,
\begin{eqnarray*}
0 &< &C_1(\delta) \le\dot\Phi^{(\delta)}(r) \le C_2(
\delta) < \infty,
\\
\ddot\Phi^{(\delta)}(r)&\le& C_2(\delta) < \infty.
\end{eqnarray*}
\end{longlist}
In particular $\Phi^{(\delta)}(r) = \int_0^r \dot\Phi^{(\delta
)}(s) \,ds \le
C_2(\delta)r$. We further choose smooth approximations $z^{(\delta)}
\in
C^\infty([0,T];\R^N)$ of the driving signal $z$. Using the homogeneity
of $\Phi$ we can rewrite \eqref{eqnroughPDEtransformed} as
%
\begin{equation}
\label{eqnroughPDEtransformed2} \partial_t Y_t =
e^{\mu_t} \D \bigl(\Phi\bigl(e^{-\mu_t}\bigr) \Phi (Y_t)
\bigr) \qquad\mbox{on } \mcO_T.
\end{equation}
One advantage of rewriting \eqref{eqnroughPDEtransformed} in this
form prior to approximating $\Phi$ by $\Phi^{(\delta)}$ is that the
substitution $Z^{(\delta)}:=\Phi^{(\delta)}(Y^{(\delta)})$ can
still be used in the
approximating equation so that the continuity results obtained in \cite
{DB83} can be applied. We construct a solution to \eqref
{eqnroughPDEtransformed2} by considering approximating equations
%
\begin{eqnarray}
\label{eqnapprox} \partial_t Y^{(\delta)}_t &=&
e^{\mu_t} \D \bigl( \Phi\bigl(e^{-\mu_t}\bigr) \Phi^{(\delta
)}
\bigl(Y^{(\delta)}_t\bigr) \bigr)\qquad \mbox{on } \mcO_T,
\nonumber
\\[-8pt]
\\[-8pt]
\nonumber
Y^{(\delta)}(0) &=& Y_0\qquad \mbox{on } \mcO,
\end{eqnarray}
with homogeneous Dirichlet boundary conditions and smooth signals $z
\in  C^\infty([0,T];\R^N)$. Equation \eqref{eqnapprox} is a
quasilinear, uniformly parabolic equation with smooth coefficients.
From standard results the unique existence of a classical solution
follows; cf., for example, \cite{LSU67}, Theorem 6.2, page 457.

\subsubsection{\texorpdfstring{Uniform $L^\infty(\mcO_T)$ bound for classical solutions to \protect\eqref{eqnapprox}}
{Uniform L infinity (O T) bound for classical solutions to (3.6)}}\label{ssecL-infty-bound}
%
\begin{lemma}\label{lemmaL-infty-boundclassicalsolns}
Let $Y_0 \in L^\infty(\mcO)$, $\{z^{(\ve)} \in C^\infty([0,T];\R
^N)| \ve>0\}$ be a compact set in $C([0,T];\R^N)$ and $Y^{(\delta,\ve)}$
be a
classical solution to \eqref{eqnapprox} driven by $z^{(\ve)}$. There
are constants $ \s_0 = \s_0(\|Y_0\|_{L^\infty(\mcO)}) > 0$, $M > 0$
depending only on $\|Y_0\|_{L^\infty(\mcO)}$, the uniform bound and
uniform modulus of continuity of $\{z^{(\ve)}\}$, piecewise smooth maps
$K^{(\s_0,\ve)}$ and a $\delta_0=\delta_0(\sup_{\ve> 0}\|
z^{(\ve)}\|
_{L^\infty(\mcO_T)},\break \|Y_0\|_{L^\infty(\mcO)})>0$ such that
\[
Y^{(\delta,\ve)} \le K^{(\s_0,\ve)} \le M \qquad\mbox{on } [0,T] \times\mcO
\]
for all $\delta\le\delta_0$.
\end{lemma}
\begin{pf}
We will construct a piecewise smooth (thus bounded) supersolution to
%
\begin{eqnarray}
\label{eqnapproxforsuper} \partial_t Y^{(\delta,\ve)}_t
&= e^{\mu^{(\ve)}_t} \D \bigl(\Phi \bigl(e^{-\mu
^{(\ve)}_t}\bigr)
\Phi^{(\delta)}\bigl(Y^{(\delta,\ve)}_t\bigr) \bigr)\qquad \mbox{on }
\mcO_{T},
\end{eqnarray}
with initial condition $Y_0$ and homogeneous Dirichlet boundary
conditions. Let $R > 0$ such that $\bar\mcO\subseteq B_R(0)$. Since
$\{
z^{(\ve)}\}$ is a set of equicontinuous functions, there exists a $\g>
0$ and a partition $0=\tau_0 < \tau_1 < \cdots < \tau_L = T$ with $1 >
\tau
_i-\tau_{i-1} > \g$ (hence $L \le\frac{T}{\g}$) such that
%
\begin{eqnarray}
\label{eqnchoicetau} \frac{1}{2} &\le& \Bigl( \mathop{\inf
_{\xi\in\bar\mcO}}_{ t \in[\tau_i,\tau
_{i+1})} e^{\mu^{(\ve)}_t-\mu^{(\ve)}_{\tau_i}}\Phi
\bigl(e^{\mu^{(\ve
)}_{\tau
_i}-\mu^{(\ve)}_t}\bigr) \Bigr)
\nonumber\\
&&{} \times\biggl( 1 - \frac{m R}{d} \sup_{t \in[\tau_i,\tau_{i+1})} \biggl( 2\bigl\|
\nabla\bigl(\mu^{(\ve)}_{\tau_i}-\mu^{(\ve)}_t
\bigr)\bigr\|_{L^\infty(\mcO)}
\nonumber
\\[-8pt]
\\[-8pt]
\nonumber
&&\hspace*{103pt}{} + \frac
{Rm}{2} \bigl\|\nabla\bigl(\mu^{(\ve)}_{\tau_i}-
\mu^{(\ve)}_t\bigr)\bigr\|_{L^\infty
(\mcO
)}^2
\\
&&\hspace*{112pt}{} + \frac{R}{2} \bigl\|\D\bigl(\mu^{(\ve)}_{\tau_i}-
\mu^{(\ve
)}_t\bigr)\bigr\| _{L^\infty(\mcO)} \biggr) \biggr)\nonumber
\end{eqnarray}
and
\[
\frac{1}{2} \le\mathop{\inf_{\xi\in\bar\mcO,}}_{ t \in[\tau_i,\tau_{i+1})}
e^{(m-1)(\mu
^{(\ve)}_t-\mu^{(\ve)}_{\tau_i})}
\]
for all $i=0,\ldots,L-1$, $\ve> 0$. Let $A^{{(m-1)}/{m}}:= \frac
{R^{{2}/{m}}}{(m-1)d}$, $C_4:= \inf_{\xi\in\mcO}(R^2-|\xi|^2)$ and
consider the inverse $\b:= \Phi^{-1}$. For $\s>0$ we define
\begin{eqnarray*}
K_0^{(\s,\ve)}(t,\xi):= \b \bigl( A (t + \s)^{-{m}/{(m-1)}}
\bigl(R^2-|\xi|^2\bigr) \Phi\bigl(e^{\mu
^{(\ve)}_{\tau_i}}\bigr)
\bigr)
\end{eqnarray*}
and choose $\s_0 = \s_0(\|Y_0\|_{L^\infty(\mcO)})$ so that $\|Y_0\|
_{L^\infty(\mcO)} \le K_0^{(\s_0,\ve)}(0)$. Then inductively define
$\s
_{i+1} = \frac{1}{2} (\s_i+\g)$ for $i=0,\dots,L-1$ (we can thus regard
$\s_i$ as a function of $\s_0$) and let
%
\begin{eqnarray}
\label{eqndefK} K_i^{(\s_0,\ve)}(t,\xi) &:=& \b \bigl( A (t-
\tau_i + \s_i)^{-{m}/{(m-1)}}\bigl(R^2-|
\xi|^2\bigr) \Phi \bigl(e^{\mu^{(\ve)}_{\tau_i}}\bigr) \bigr)
\nonumber
\\[-8pt]
\\[-8pt]
\nonumber
&=& A^{{1} /{m}} (t-\tau_i + \s_i)^{-{1}/{(m-1)}}
\bigl(R^2-|\xi |^2\bigr)^{ {1} /{m}}
e^{\mu^{(\ve)}_{\tau_i}}
\end{eqnarray}
for $t \in[\tau_i,\tau_{i+1}], \xi \in\mcO.$

By the choice of $\s_i$, $i=1,\dots,L-1$ we have $K^{(\s_0,\ve
)}_i(\tau
_{i+1}) \le K^{(\s_0,\ve)}_{i+1}(\tau_{i+1})$. We note
%
\begin{eqnarray}
\label{eqnKbounds}\qquad&& A^{{1}/ {m} }\Bigl(1 + \max_{i=0,\ldots,L-1}
\s_i \Bigr)^{-
{1}/{(m-1)}} C_4^{{1}/ {m}}
e^{-\sup_{\ve> 0}\|\mu^{(\ve)}\|
_{L^\infty
(\mcO_T)}}
\nonumber
\\[-8pt]
\\[-8pt]
\nonumber
&&\qquad\le K^{(\s_0,\ve)}_i(t) \le A^{{1} /{m}} \Bigl(
\min_{i=0,\ldots,L-1}\s_i \Bigr)^{-
{1}/{(m-1)}}
R^{{2} /{m}} e^{\sup_{\ve> 0}\|\mu^{(\ve)}\|_{L^\infty
(\mcO_T)}}
\end{eqnarray}
for all $t \in[\tau_i,\tau_{i+1}]$. Hence, we can choose $\delta_0 >0$
(depending only on $\s_0$, $\sup_{\ve> 0}\|z^{(\ve)}\|_{L^\infty
(\mcO
_T)}$) such that
\[
\label{eqnve0} K_i^{(\s_0,\ve)}(t) \in\biggl[\delta,
\frac{1}{\delta}\biggr]
\]
for all $t \in[\tau_i,\tau_{i+1}]$ and $\delta\le\delta_0$. Then
$\Phi^{(\delta
)}(K_i(t)) = \Phi(K_i(t))$, and we compute (for simplicity we drop the
$\ve$ dependencies and the $\s_0$ dependency of $K_i$)
\begin{eqnarray*}
&&\D \bigl( \Phi\bigl(e^{-\mu_t}\bigr) \Phi^{(\delta)}
\bigl(K_i(t)\bigr) \bigr)
\\
&&\qquad= \D \bigl( A (t-\tau_i + \s_i)^{-{m}/{(m-1)}}
\bigl(R^2-|\xi|^2\bigr) \Phi \bigl(e^{\mu_{\tau_i}-\mu_t}\bigr)
\bigr)
\\
&&\qquad= A (t- \tau_i + \s_i)^{-{m}/{(m-1)}} \Phi
\bigl(e^{\mu_{\tau_i}-\mu_t}\bigr)
\\
&&\quad\qquad{} \times\bigl( -2d - 4m \xi\cdot\nabla(\mu_{\tau_i}-\mu_t) +
\bigl(R^2-|\xi|^2\bigr)\\
&&\hspace*{56pt}{}\times \bigl(m^2 \bigl|\nabla(
\mu_{\tau_i}-\mu_t)\bigr|^2 + m \D(\mu _{\tau
_i}-
\mu_t)\bigr) \bigr)
\end{eqnarray*}
and
\begin{eqnarray*}
\partial_t K_i(t) &= - \frac{A^{{1}/{m}}}{m-1} (t-
\tau_i + \s_i)^{-
{m}/{(m-1)}}\bigl(R^2-|
\xi|^2\bigr)^{{1}/ {m}} e^{\mu_{\tau_i}}.
\end{eqnarray*}
In order to show that $K_i(t)$ is a supersolution to \eqref
{eqnapproxforsuper} on $[\tau_i,\tau_{i+1}]$, we thus have to show
\begin{eqnarray*}
0& \le&\partial_t K_i(t) - e^{\mu_t} \D \bigl(
\Phi\bigl(e^{-\mu_t}\bigr) \Phi ^{(\delta
)}\bigl(K_i(t)
\bigr) \bigr)
\\
&=& - \frac{A^{{1}/{m}}}{m-1} (t- \tau_i + \s_i)^{-
{m}/{(m-1)}}
\bigl(R^2-|\xi|^2\bigr)^{{1}/ {m}}
e^{\mu_{\tau_i}} \\
&&{}- A (t- \tau_i + \s _i)^{-{m}/{(m-1)}}
e^{\mu_t}\Phi\bigl(e^{\mu_{\tau_i}-\mu_t}\bigr)
\\
&&\quad{}\times \bigl( -2d - 4m \xi\cdot\nabla(\mu_{\tau_i}-\mu_t) +
\bigl(R^2-|\xi|^2\bigr) \\
&&\quad\hspace*{14pt}{}\times \bigl( m^2 \bigl|\nabla(
\mu_{\tau_i}-\mu_t)\bigr|^2 + m \D (\mu
_{\tau_i}-\mu_t) \bigr) \bigr)
\end{eqnarray*}
for all $t \in[\tau_i,\tau_{i+1}]$. Equivalently,
\begin{eqnarray*}
\frac{ (R^2-|\xi|^2)^{{1}/{m}}}{m-1} &\le &A^{{(m-1)}/{m}} e^{\mu_t-\mu_{\tau_i}}\Phi
\bigl(e^{\mu_{\tau
_i}-\mu
_t}\bigr) \\
&&{}\times\bigl( 2d + 4m \xi\cdot\nabla(\mu_{\tau_i}-
\mu_t)- \bigl(R^2-|\xi|^2\bigr) \\
&&\hspace*{9pt}{}\times\bigl( m^2 \bigl|
\nabla(\mu_{\tau_i}-\mu_t)\bigr|^2 + m \D(
\mu_{\tau_i}-\mu_t) \bigr) \bigr).
\end{eqnarray*}
It is thus sufficient to show
\begin{eqnarray*}
\frac{ R^{{2}/{m} }}{m-1} &\le& A^{{(m-1)}/{m}} \Bigl( \mathop{\inf
_{\xi\in\bar\mcO}}_{ t
\in[\tau_i,\tau_{i+1}]} e^{\mu_t-\mu_{\tau_i}}\Phi
\bigl(e^{\mu_{\tau
_i}-\mu
_t}\bigr) \Bigr)\\
&&{}\times \bigl( 2d - 4m R \bigl\| \nabla(
\mu_{\tau_i}-\mu_t)\bigr\| _{L^\infty
(\mcO)}
 \\
 &&\hspace*{15pt}{}- R^2 \bigl( m^2 \bigl\|\nabla(\mu_{\tau_i}-
\mu_t)\bigr\| _{L^\infty
(\mcO)}^2 + m \bigl\|\D(\mu_{\tau_i}-
\mu_t)\bigr\|_{L^\infty(\mcO)} \bigr) \bigr)
\end{eqnarray*}
for all $t \in[\tau_i,\tau_{i+1}]$, which is satisfied by the choice
of $A$ and $\tau_i$ in \eqref{eqnchoicetau}. In conclusion,
$K_i^{(\s
_0,\ve)}(t)$ is a supersolution to \eqref{eqnapproxforsuper} on
$[\tau_i,\tau_{i+1}]$ for each $\delta\le\delta_0$. We define
%
\begin{equation}
\label{eqndefKs} K^{(\s_0,\ve)}(t):= \sum_{i=0}^{L-1}
\mathbh{1}_{[\tau_i,\tau
_{i+1})}(t) K^{(\s_0,\ve)}_i(t).
\end{equation}
Since the comparison principle \cite{L96}, Theorem 9.7, applies on each
interval $[\tau_i,\tau_{i+1}]$, by induction we have
\[
Y^{(\delta,\ve)}(t,\xi) \le K^{(\s_0,\ve)}(t,\xi) \qquad\forall t \in [0,T], \xi\in
\mcO, \delta\le\delta_0.
\]
The upper bound in \eqref{eqnKbounds} yields a uniform bound $M$ for
$K^{(\s_0,\ve)}$. $M$ depends on $\s_0$, $\sup_{\ve}\|z^{(\ve)}\|
_{L^\infty(\mcO)}$ and via the bound of the partition size $\gamma$ and
the definition of $\s_i$, on the uniform modulus of continuity of $\{
z^{(\ve)}\}$.
\end{pf}

\subsubsection{Existence of weak solutions}\label{ssecweaksoln}
We will now take the limit $\delta\to0$ in \eqref{eqnapprox} in order
to obtain weak solutions to \eqref{eqnroughPDEtransformed} in the
sense of Definition \ref{defweaksoln}.

\begin{lemma}\label{lemmaapriori-subgradient}
Let $Y_0 \in L^\infty(\mcO)$, $\{z^{(\ve)} \in C^\infty([0,T];\R
^N)| \ve>0\} \subseteq C([0,T];\break  \R^N)$ be compact and $Y^{(\delta,\ve)}$
be a
classical solution to \eqref{eqnapprox} driven by $z^{(\ve)}$. Then
%
\begin{eqnarray}
\label{eqnapriorifordapprox}\qquad&& \sup_{t \in[0,T]} \bigl(
\bigl\|Y^{(\delta,\ve)}_t\bigr\|_{m+1}^{m+1} + \bigl\|
Y_t^{(\delta
,\ve)}\bigr\|_H^2 \bigr)
+C_1 \bigl\|\nabla \bigl( \Phi\bigl(e^{-\mu^{(\ve)}}\bigr)
\Phi^\delta \bigl(Y^{(\delta,\ve)}\bigr) \bigr)\bigr \|_{L^2(\mcO_T)}
\nonumber
\\[-8pt]
\\[-8pt]
\nonumber
&&\qquad \le
C_2
\end{eqnarray}
for all $\ve>0$, $\delta\le\delta_0$ (with $\delta_0$ from Lemma
\ref
{lemmaL-infty-boundclassicalsolns}) and for some constants $0 < C_1,
C_2$ independent of $\delta$ and $\ve$. $C_2$ may depend on $\|Y_0\|
_{L^\infty(\mcO)}$, the uniform bound and the uniform modulus of
continuity of $\{z^{(\ve)}\}$.
\end{lemma}
\begin{pf}
Let $\Psi^{(\delta)} \in C^1(\R)$ so that $\dot\Psi^{(\delta)} =
\Phi^{(\delta
)}$. We compute
%
\begin{eqnarray}
\label{eqnsubgradienttest} &&\hspace*{-4pt}\partial_t \int_\mcO
\Psi^{(\delta)}\bigl(Y^{(\delta,\ve)}_t\bigr) \,d\xi
\nonumber
\\
&&\hspace*{-4pt}\qquad =
\int_\mcO\frac{\Phi(e^{-\mu^{(\ve)}_t})}{\Phi(e^{-\mu^{(\ve
)}_t})} \Phi^{(\delta)}
\bigl(Y^{(\delta,\ve)}_t\bigr) \partial_t Y^{(\delta
,\ve)}
\,d\xi
\nonumber
\\
&&\hspace*{-4pt}\qquad = - \int_\mcO\frac{e^{\mu^{(\ve)}_t}}{\Phi(e^{-\mu^{(\ve)}_t})} \nabla \bigl( \Phi
\bigl(e^{-\mu^{(\ve)}_t}\bigr) \Phi^{(\delta)}\bigl(Y^{(\delta
,\ve)}_t
\bigr) \bigr) \nabla \bigl(\Phi\bigl(e^{-\mu^{(\ve)}_t}\bigr) \Phi^{(\delta
)}
\bigl(Y^{(\delta,\ve
)}_t\bigr) \bigr) \,d\xi
\nonumber
\\[-8pt]
\\[-8pt]
\nonumber
&&\hspace*{-4pt}\qquad\quad{} - \int_\mcO\Phi\bigl(e^{-\mu^{(\ve)}_t}\bigr)
\Phi^{(\delta
)}\bigl(Y^{(\delta,\ve
)}_t\bigr) \nabla \biggl(
\frac{e^{\mu^{(\ve)}_t}}{\Phi(e^{-\mu^{(\ve
)}_t})} \biggr) \nabla \bigl( \Phi\bigl(e^{-\mu^{(\ve)}_t}\bigr) \Phi
^{(\delta)}\bigl(Y^{(\delta
,\ve)}_t\bigr) \bigr) \,d\xi
\\
&&\hspace*{-4pt}\qquad\le\sup_{(t,\xi) \in\bar\mcO_T} \biggl(\ve_1 \biggl\llvert \nabla
\frac
{e^{\mu^{(\ve)}_t}}{\Phi(e^{-\mu^{(\ve)}_t})}\biggr\rrvert ^2 -\frac
{e^{\mu
^{(\ve)}_t}}{\Phi(e^{-\mu^{(\ve)}_t})} \biggr) \int
_\mcO\bigl|\nabla \Phi \bigl(e^{-\mu^{(\ve)}_t}\bigr)
\Phi^{(\delta)}\bigl(Y^{(\delta,\ve)}_t\bigr)\bigr|^2 \,d\xi
\nonumber
\\
&&\hspace*{-4pt}\qquad\quad{} + C_{\ve_1} \int_\mcO \bigl( \Phi
\bigl(e^{-\mu^{(\ve
)}_t}\bigr) \Phi ^{(\delta)}\bigl(Y^{(\delta,\ve)}_t
\bigr) \bigr)^2 \,d\xi
\nonumber
\end{eqnarray}
for all $\ve_1 > 0$ and some $C_{\ve_1} > 0$. Choosing $\ve_1$ small
enough and using the uniform $L^\infty$ bound derived in Lemma \ref
{lemmaL-infty-boundclassicalsolns}, we conclude
\begin{eqnarray*}
&&\sup_{t \in[0,T]} \int_\mcO\Psi^{(\delta)}
\bigl(Y^{(\delta,\ve
)}_t\bigr) \,d\xi+ C_1 \int
_{\mcO_T} \bigl|\nabla\Phi\bigl(e^{-\mu^{(\ve)}_r}\bigr)
\Phi^{(\delta
)}\bigl(Y^{(\delta,\ve
)}_r\bigr)\bigr|^2 \,d\xi
\,dr
\\
&&\qquad\le\int_\mcO\Psi^{(\delta)}(Y_0) \,d\xi+
C_2
\end{eqnarray*}
for all $\delta\le\delta_0$ and for some constants $C_1, C_2 > 0$ independent
of $\delta$ and $\ve$, where $C_2$ may depend on $\|Y_0\|_{L^\infty
(\mcO
)}$, the uniform bound and the uniform modulus of continuity of $\{
z^{(\ve)}\}$.

It remains to prove the bound of $\|Y^{(\delta)}\|_H^2$. By the chain rule
we have
\[
\frac{d}{dt} \bigl\|Y_t^{(\delta,\ve)}\bigr\|_H^2
= 2 \int_\mcO(-\D)^{-1} \bigl(Y_t^{(\delta,\ve)}
\bigr) e^{\mu
^{(\ve
)}_t}\D \bigl(\Phi\bigl(e^{-\mu^{(\ve)}_t}\bigr)
\Phi^{(\delta)}\bigl(Y^{(\delta
,\ve)}_t\bigr) \bigr) \,d\xi.
\]
Since for $f,g,h$ sufficiently smooth and $h_{|\partial\mcO} = 0$ we have
\begin{eqnarray*}
\int_\mcO fg \D h \,d\xi= \int_\mcO
\bigl( f \D(gh) + 2 h \nabla f \cdot\nabla g + fh \D(g) \bigr) \,d\xi.
\end{eqnarray*}
We obtain
%
\begin{eqnarray}
\label{eqnapriorifordapproxH} &&\frac{d}{dt} \bigl\|Y_t^{(\delta,\ve)}
\bigr\|_H^2
\nonumber
\\
&&\qquad= - 2 \int_\mcO Y_t^{(\delta,\ve)}
e^{\mu^{(\ve)}_t}\Phi \bigl(e^{-{\mu^{(\ve
)}_t}}\bigr)\Phi^{(\delta)}
\bigl(Y^{(\delta,\ve)}_t\bigr) \,d\xi
\nonumber
\\
&&\qquad\quad{}+ 2\int_\mcO\Phi\bigl(e^{-\mu^{(\ve)}_t}\bigr)
\Phi^{(\delta
)}\bigl(Y^{(\delta,\ve
)}_t\bigr) \bigl(2\nabla
\bigl((-\D)^{-1}\bigl(Y_t^{(\delta,\ve)}\bigr) \bigr) \cdot
\nabla \bigl( e^{\mu^{(\ve)}_t} \bigr) \bigr) \,d\xi
\\
&&\qquad\quad{} + 2\int_\mcO\Phi\bigl(e^{-\mu^{(\ve)}_t}\bigr)
\Phi^{(\delta
)}\bigl(Y^{(\delta,\ve
)}_t\bigr) \bigl((-
\D)^{-1}\bigl(Y_t^{(\delta,\ve)}\bigr) \D
\bigl(e^{\mu^{(\ve
)}_t} \bigr) \bigr) \,d\xi\nonumber
\\
&&\qquad\le C \bigl(1+ \bigl\|Y_t^{(\delta,\ve)}\bigr\|_H^2
\bigr)\qquad\delta\le \delta_0,
\nonumber
\end{eqnarray}
where $0 < C $ is a constant independent of $\delta$, $\ve$, possibly
depending on\break  $\|Y_0\|_{L^\infty(\mcO)}$, the uniform bound and the
uniform modulus of continuity of $\{z^{(\ve)}\}$. Gronwall's inequality
then yields the bound.
\end{pf}

\begin{pf*}{Proof of Theorem \ref{thmexistenceweak}}
We approximate the initial condition $Y_0$ by smooth functions
$Y_0^{(\delta)} \in C^2(\bar\mcO)$ such that $Y_0^{(\delta)} \to
Y_0$ almost
everywhere and $\|Y_0^{(\delta)}\|_{L^\infty(\mcO)} \le\|Y_0\|
_{L^\infty
(\mcO)}$. The continuous driving signal $z$ is approximated by smooth
signals $z^{(\delta)} \in C^\infty([0,T];\R^N)$ such that
$z^{(\delta)} \to z$
in $C([0,T];\R^N)$. In particular $\{z^{(\delta)}| \delta> 0\}$ is
a compact
set in $C([0,T];\R^N)$. Let $Y^{(\delta)}$ be classical solutions to
\eqref
{eqnapprox} with initial condition $Y_0^{(\delta)}$ and driving signal
$z^{(\delta)}$. In the following let $\delta\le\delta_0$ with
$\delta_0$ as in Lemma
\ref{lemmaapriori-subgradient}.

By Lemma \ref{lemmaapriori-subgradient}, $Y^{(\delta)}$ is uniformly
bounded in $L^\infty([0,T];L^{m+1}(\mcO))$ and in $L^\infty([0,T];H)$.
By Sobolev embedding, for $k \ge\frac{n}{2} (\frac{1-m}{1+m}) \vee1$
we have $H^k_0(\mcO) \hookrightarrow L^{{(m+1)}/{m}}(\mcO)$.
Consequently, $L^{m+1}(\mcO) \hookrightarrow H^{-k}:= (H^k_0(\mcO))^*$
and $H \hookrightarrow H^{-k}$. Hence, weak$^*$ limits obtained in
$L^\infty([0,T];L^{m+1}(\mcO))$ and $L^\infty([0,T];H)$ coincide.

Moreover, $\Phi(e^{-\mu^{(\delta)}_t})\Phi^{(\delta)}(Y^{(\delta
)})$ is uniformly
bounded in $L^2([0,T];H_0^1(\mcO))$ and boundedness of $Y^{(\delta)}$ in
$L^\infty([0,T];L^{m+1}(\mcO))$ implies boundedness of $\Phi(e^{-\mu
^{(\delta)}_t})\Phi^{(\delta)}(Y^{(\delta)})$ in $L^\infty
([0,T];L^{
{(m+1)}/{m}}(\mcO))$.\vadjust{\goodbreak}

Hence, we can choose a subsequence (again denoted by $\delta$) such that
\begin{eqnarray*}
Y^{(\delta)} &\rightharpoonup^*& Y\qquad \mbox{in } L^\infty
\bigl([0,T];L^{m+1}(\mcO )\bigr)
\mbox{ and in } L^\infty\bigl([0,T];H\bigr),
\\
Z^{(\delta)} &:=& \Phi\bigl(e^{-\mu^{(\delta)}_t}\bigr)\Phi^{(\delta
)}
\bigl(Y^{(\delta)}\bigr) \rightharpoonup Z \qquad\mbox{in } L^2
\bigl([0,T];H_0^1(\mcO)\bigr),
\\
Z^{(\delta)} &\rightharpoonup^* &Z \qquad\mbox{in } L^\infty
\bigl([0,T];L^{{(m+1)}/{m}}(\mcO)\bigr).
\end{eqnarray*}
Since
\begin{eqnarray*}
\label{eqnapproxweaksoln} && -\int_{\mcO_T} Y_r^{(\delta)}
\partial_r \eta_r \,d\xi \,dr - \int_\mcO
Y^{(\delta)}_0 \eta_0 \,d\xi
\\
&&\qquad= - \int_{\mcO_T} \nabla \bigl(\Phi\bigl(e^{-\mu^{(\delta)}_r}
\bigr) \Phi ^{(\delta
)}\bigl(Y_r^{(\delta)}\bigr) \bigr)
\nabla \bigl(e^{\mu^{(\delta)}_r}\eta _r \bigr) \,d\xi \,dr,
\end{eqnarray*}
we obtain
\[
\label{eqnlimitweaksoln} -\int_{\mcO_T} Y_r
\partial_r \eta \,d\xi \,dr - \int_\mcO
Y_0 \eta _0 \,d\xi= - \int_{\mcO_T}
\nabla Z_r \nabla \bigl( e^{\mu_r} \eta_r \bigr)
\,d\xi \,dr
\]
for all $\eta\in C^{1}(\bar\mcO_T)$ with $\eta= 0$ on $\mcP\mcO_T$.

First we will prove that $Y^{(\delta)}_t \rightharpoonup Y_t$ in $H$, for
all $t \in[0,T]$. We consider the set $\mcK= \{(Y^{(\delta)},h)_H| h \in
H, \|h\|_H \le1, \delta> 0\} \subseteq C([0,T])$. By Lemma \ref
{lemmaapriori-subgradient}, $\mcK$ is bounded in $C([0,T])$. Moreover,
\begin{eqnarray*}
\bigl(Y^{(\delta)}_{t+s}-Y^{(\delta)}_t,h
\bigr)_H &=& \int_t^{t+s} \biggl(
\frac{dY^{(\delta)}}{dr},h \biggr)_H \,dr \le\|h\|_H
s^{{1}/ {2}} \biggl\llVert \frac{dY^{(\delta)}}{dr} \biggr\rrVert
_{L^2([0,T];H)}
\\
&\le& C \|h\|_H s^{{1}/ {2}}.
\end{eqnarray*}
Hence, $\mcK$ is a set of equibounded, equicontinuous functions and
thus is relatively compact in $C([0,T])$. For every $h \in H, \|h\|_H
\le1 $ there is a subsequence (again denoted by $\delta$) such that
$(Y^{(\delta)},h)_H \to g$ in $C([0,T])$. Since also $Y^{(\delta)}
\rightharpoonup Y$ in $L^2([0,T];H)$ (thus $(Y^{(\delta)},h)_H
\rightharpoonup(Y,h)_H$ in $L^2([0,T])$) we have $g = (Y,h)$ which
implies $Y^{(\delta)}_t \rightharpoonup Y_t$ in $H$ for all $t\in[0,T]$.

We need to prove $Z = \Phi(e^{-\mu} Y)$ almost everywhere. This will
be done by considering the equation on $H = (H_0^1(\mcO))^*$. Since
$Y^{(\delta)}$ solves \eqref{eqnapprox}, we conclude that
\begin{eqnarray*}
\frac{dY^{(\delta)}}{dt} \rightharpoonup\frac{dY}{dt}\qquad \mbox{in }
L^2\bigl([0,T];H\bigr)
\end{eqnarray*}
and
\begin{eqnarray*}
\frac{dY}{dt} &=& e^{\mu_t} \D Z \qquad\mbox{for a.e. } t \in[0,T],
\\
Y(0) &=& Y_0.
\end{eqnarray*}
In particular, since also $Y \in L^\infty([0,T];H)$ we have $Y \in
C([0,T];H)$. By the chain rule we obtain
%
\begin{eqnarray}
\label{eqnlimitIto} \|Y_t\|_H^2& = &
\|Y_0\|_H^2 - 2 \int_0^t
\int_{\mcO} e^{\mu_r} Z_r Y_r
\,d\xi \,dr
\nonumber
\\
&&{} + 2 \int_0^t \int_{\mcO}
Z_r \bigl( 2 \nabla\bigl(e^{\mu_r}\bigr) \nabla \bigl((-\D
)^{-1}(Y_r)\bigr)\\
&&\hspace*{73pt}{}+ \D\bigl(e^{\mu_r}\bigr) (-
\D)^{-1}(Y_r) \bigr) \,d\xi \,dr.\nonumber
\end{eqnarray}
Applying the chain rule to \eqref{eqnapprox} yields
%
\begin{eqnarray}
\label{eqnItoapprox} \bigl\|Y^{(\delta)}_t\bigr\|_H^2
&=& \bigl\|Y^{(\delta)}_0\bigr\|_H^2 - 2 \int
_0^t \int_{\mcO}
e^{\mu^{(\delta)}_r} Z_r^{(\delta)} Y_r^{(\delta)}
\,d\xi \,dr
\nonumber
\\
&&{}+ 2 \int_0^t \int_{\mcO}
Z_r^{(\delta)} \bigl( 2 \nabla\bigl(e^{\mu
^{(\delta)}_r}\bigr)
\nabla\bigl((-\D)^{-1}\bigl(Y_r^{(\delta)}\bigr)\bigr) \\
&&\hspace*{81pt}{}+ \D\bigl(e^{\mu^{(\delta)}_r}\bigr) (-\D)^{-1}\bigl(Y_r^{(\delta)}
\bigr) \bigr) \,d\xi \,dr.\nonumber
\end{eqnarray}
Since $(-\D)^{-1}(Y^{(\delta)}) \in L^2([0,T];H_0^1(\mcO))$ and
\[
\frac{d(-\D)^{-1}(Y^{(\delta)})}{dt} \in L^2\bigl([0,T];H_0^1(
\mcO)\bigr) \subseteq L^2\bigl([0,T];L^2(\mcO)\bigr)
\]
are uniformly bounded and $H_0^1(\mcO) \hookrightarrow\hookrightarrow
L^2(\mcO) $, by the Aubin--Lions compactness theorem we have (for a
subsequence again denoted by $\delta$)
\[
(-\D)^{-1}\bigl(Y^{(\delta)}\bigr) \to(-\D)^{-1}(Y)\qquad
\mbox{strongly in } L^2\bigl([0,T];L^2(\mcO)\bigr).
\]
Note that also $Z^{(\delta)} \rightharpoonup Z$ in
$L^2([0,T];H_0^1(\mcO
))$. Taking the limit $\delta\to0$ in \eqref{eqnItoapprox} yields
\begin{eqnarray*}
\|Y_t\|_H^2 &\le& \|Y_0
\|_H^2 - \limsup_{\delta\to0} 2 \int
_0^t \int_{\mcO}
e^{\mu
^{(\delta)}_r} Z_r^{(\delta)} Y_r^{(\delta)}
\,d\xi \,dr
\\
&&{}+ 2 \int_0^t \int_{\mcO}
Z_r \bigl( 2 \nabla\bigl(e^{\mu_r}\bigr) \nabla \bigl((-\D
)^{-1}(Y_r)\bigr) + \D\bigl(e^{\mu_r}\bigr) (-
\D)^{-1}(Y_r) \bigr) \,d\xi \,dr.
\nonumber
\end{eqnarray*}

Substracting \eqref{eqnlimitIto} we arrive at
%
\begin{equation}
\label{eqnmonotonicitytrick} \limsup_{\delta\to0} \int
_{\mcO_T} e^{\mu^{(\delta)}_r} Z_r^{(\delta)}
Y_r^{(\delta
)}\,d\xi \,dr \le\int_{\mcO_T}
e^{\mu_r} Z_r Y_r \,d\xi \,dr.
\end{equation}

By monotonicity of $\Phi^{(\delta)}$ we have
\[
\int_{\mcO_T} e^{\mu^{(\delta)}_r}\Phi\bigl(e^{-\mu^{(\delta)}_r}
\bigr) \bigl(\Phi^{(\delta
)}\bigl(Y_r^{(\delta)}\bigr) -
\Phi^{(\delta)}(z_r)\bigr) \bigl(Y_r^{(\delta)}-z_r
\bigr) \,d\xi \,dr \ge0
\]
for all $z \in C^1(\bar\mcO_T)$. Using \eqref{eqnmonotonicitytrick}
we can take $\delta\to0$ to obtain
\[
\int_{\mcO_T} e^{\mu_r} \bigl(Z_r - \Phi
\bigl(e^{-\mu_r}\bigr)\Phi (z_r)\bigr) (Y_r-z_r)
\,d\xi \,dr \ge0
\]
for all $z \in C^1(\bar\mcO_T)$, hence by approximation for all $z
\in
L^{m+1}(\mcO_T)$. Taking $z=Y-\ve h$ with $h\in C^0(\bar\mcO_T)$,
dividing by $\ve$ and letting $\ve\to0$ yields
\[
\int_{\mcO_T} e^{\mu_r} \bigl(Z_r - \Phi
\bigl(e^{-\mu_r}\bigr)\Phi(Y_r)\bigr)h \,d\xi \,dr \ge0
\]
for all $h\in C^0(\mcO_T)$. This implies $Z = \Phi(e^{-\mu})\Phi(Y)$
almost everywhere.

It remains to prove that the uniform $L^\infty$ bound obtained in
Lemma \ref{lemmaL-infty-boundclassicalsolns} remains valid for weak
solutions. We first note that by uniform continuity of $\{z^{(\delta
)}| \delta
> 0\}$ the partition $\tau_i$ in \eqref{eqnchoicetau} can be chosen
independently of $\delta$. Thus $K_i^{(\s_0,\delta)}$ defined in~\eqref
{eqndefK} only depends on $\delta$ via the factor $e^{\mu^{(\delta
)}_{\tau
_i}}$ and converges uniformly to a piecewise smooth function $K_i^{(\s
_0)}$ given by \eqref{eqndefK} with $\mu^{(\ve)} = \mu$. We define
$K^{(\s_0)}$ as in \eqref{eqndefKs}. By Lemma \ref
{lemmaL-infty-boundclassicalsolns} we know that $Y^{(\delta)}_t \le
K^{(\delta,\s_0)}(t)$ for all $t \in[0,T]$ and all $\delta\le
\delta_0$. Since
the cone of nonnegative distributions in $H$ is convex, closed and
$Y^{(\delta)}_t \rightharpoonup Y_t$ in $H$ we conclude $Y_t \le
K^{(\s
_0)}(t)$ a.e. in $\mcO$ for all $t \in[0,T]$. Note that $K^{(\s_0)}$
is increasing as $\s_0$ decreases. Defining $U:= K^{(0)}\dvtx[0,T]
\to
\bar\R$ as in \eqref{eqndefK} with $\s_0 = 0$ (with the convention
$\frac{1}{0} = \infty$) yields a piecewise smooth function on $(0,T]$
(taking the value $\infty$ at $t=0$) with $Y_t \le K^{(\s_0)}(t) \le
U_t$ a.e. in $\mcO$ and for all $t\in[0,T]$.

For later use we prove weak continuity of $t \mapsto Y_t$ in $L^p(\mcO
)$. Let $p \in(2,\infty)$ and $t_n \to t \in[0,T]$. Then $Y_t$ is
uniformly bounded in $L^p(\mcO)$ and thus there is a weakly convergence
subsequence $Y_{t_{n_k}}$. Since $Y \in C([0,T];H)$, the weak limit is
$Y_t$ and by arbitrarity of the sequence $t_n$ we obtain $Y_{t_n}
\rightharpoonup Y_t$ in $L^p(\mcO)$.

Assume that $z \in C^{1-\operatorname{var}}([0,T];\R^N)$, by Theorem \ref
{thmtransformationbddvarn} $X = e^{-\mu}Y$ is a weak solution to
\eqref{eqnroughPDE}, and the bounds follow from the corresponding
ones for $Y$.
\end{pf*}

\begin{pf*}{Proof of Remark \ref{rmkfde}}
The proof of existence of weak solutions to \eqref{eqnroughPDE} and
\eqref{eqnroughPDEtransformed} proceeds with only minor
modifications for the case of $0<m<1$. The statements of Lemma \ref
{lemmaL-infty-boundclassicalsolns} remain true, however, with a
modified upper bound $K^{(\s_0,\ve)}$.
\begin{pf*}{Proof of Lemma \ref{lemmaL-infty-boundclassicalsolns} for
fast diffusion equations}
Again we construct a supersolution to \eqref
{eqnapproxforsuper} which is piecewise smooth (thus bounded) in
$\bar
\mcO_T$. Let $R, \b, C_4$ and $\tau_i$, $i=0,\ldots,L-1$ as before and
$A^{{(m-1)}/{m}} = \frac{R^{{2}/{m}}}{(1-m)d}$. We inductively define
\begin{eqnarray*}
K_i^{(\s_0,\ve)}(t,\xi) = A^{{1}/ {m} }(
\s_i-t)^{
{1}/{(1-m)}}\bigl(R^2-|\xi|^2
\bigr)^{{1}/ {m}} e^{\mu^{(\ve)}_{\tau_i}},\hspace*{-2pt}\qquad t \in[\tau_i,
\tau_{i+1}], \xi\in\mcO,
\end{eqnarray*}
where $\s_i > \tau_{i+1}$, $i=1,\ldots,L-1$ are chosen (large enough) such
that\break  $K_0^{(\s_0,\ve)}(0) \ge Y_0$ and $K_i(\tau_{i+1}) \le
K_{i+1}(\tau
_{i+1})$, which is satisfied if $\s_{i+1} \ge2\s_i + \tau_{i+1}$. The
remaining calculations and arguments are similar to those of the
degenerate case. Note, however, the changing signs due to the changing
sign of $1-m$.
\end{pf*}

We now return to the proof of Remark \ref{rmkfde}. We continue by
proving a priori estimates for the approximating classical solutions
analogous to those given in Lemma \ref{lemmaapriori-subgradient}. Here
we can allow $Y_0 \in L^{m+1}(\mcO)$ since in \eqref
{eqnsubgradienttest} and \eqref{eqnapriorifordapproxH} the term
$\int_\mcO ( \Phi(e^{-\mu_t}) \Phi^{(\delta)}(Y^{(\delta)})
)^2 \,d\xi
$ can be bounded by $C\int_\mcO\Psi^{(\delta)}(Y^{(\delta)}) \,d\xi
$. Thus, the
$L^\infty$ bound is not needed to prove \eqref
{eqnapriorifordapprox}. The same proof as for Theorem \ref
{thmexistenceweak} can then be used to construct weak solutions for
all initial conditions $Y_0 \in L^{m+1}(\mcO)$ [but without $L^\infty
(\mcO)$ bound]. This finishes the proof of existence of weak solutions
for the case of fast diffusions. If $Y_0 \in L^\infty(\mcO)$, then
Lemma \ref{lemmaL-infty-boundclassicalsolns} yields $L^\infty$
boundedness of $Y$.

In order to obtain a uniform upper bound independent of the initial
condition as in the degenerate case ($m > 1$), we would have to let $\s
_0 \to\infty$ in $K^{(\s_0)}$ implying $U \equiv\infty$. Moreover, we
do not have a uniqueness result for essentially bounded weak solutions
in the case of fast diffusion equations. Therefore, it is not known
whether each such weak solution is a limit of solutions to the
nondegenerate approximating equations which will be needed for the
proof of uniform continuity in the initial condition with respect to
the $L^1$ norm.
\end{pf*}

%
%

\subsection{Rough weak solutions}\label{ssecroughsolns}
We prove Theorem \ref{thmroughPDE}. Let $Y_0 \in L^\infty(\mcO)$ and
$z^{(\ve)} \in C^{1-\operatorname{var}}([0,T];\R^N)$ such that $z^{(\ve)} \to z$ in
$C([0,T];\R^N)$. In particular $\{z^{(\ve)}| \ve>0\}$ is compact in
$C([0,T];\R^N)$. We require uniform bounds for the corresponding weak
solutions $Y^{(\ve)}$ to \eqref{eqnroughPDEtransformed} driven by
$z^{(\ve)}$.
%
\begin{lemma}\label{lemmaaprioriuniform}
Let $\{z^{(\ve)}| \ve>0\} \subseteq C([0,T];\R^N)$ compact and
$Y^{(\ve)}$ the weak solutions to \eqref{eqnroughPDEtransformed}
driven by $z^{(\ve)}$. Then there exists a constant $M >0$ (independent
of $\ve$) such that
\[
\sup_{t \in[0,T]} \bigl\|Y_t^{(\ve)}\bigr\|_{L^\infty(\mcO)}
+ \bigl\|\Phi \bigl(e^{-\mu
^{(\ve)}}Y^{(\ve)}\bigr)\bigr\|_{L^2([0,T];H_0^1(\mcO))}^2
\le M.
\]
\end{lemma}
\begin{pf}
For $\ve> 0$ let $\{z^{(\tau,\ve)} \in C^\infty([0,T];\R^N)| \tau>
0\}$ be the sequence of smooth functions obtained by convolution of
$z^{(\ve)}$ with a standard Dirac sequence. Since $\{z^{(\ve)}| \ve
>0\}$ is a set of equicontinuous functions, there is a uniform modulus
of continuity $\omega\dvtx\R_+ \to\R_+$. Uniform boundedness and the modulus
of continuity are preserved under convolution with a Dirac sequence.
Thus, the set $\{z^{(\tau,\ve)}| \ve>0, \tau> 0\}$ is compact in
$C([0,T];\R^N)$.

Let now $Y_0^{(\delta)}$ be a smooth approximation of $Y_0$ as in the
proof of Theorem~\ref{thmexistenceweak}, and let $Y^{(\delta,\ve
)}$ be
the corresponding smooth solution to \eqref{eqnapprox} driven by
$z^{(\delta,\ve)}$.
By Lemmas \ref{lemmaL-infty-boundclassicalsolns} and~\ref
{lemmaapriori-subgradient} there is a uniform constant $M>0$
(depending only on $\|Y_0\|_{L^\infty(\mcO)}$) such that\vspace*{1pt}
\[
\bigl\|Y^{(\delta,\ve)}\bigr\|_{L^\infty(\mcO)}+\bigl\|\Phi\bigl(e^{-\mu^{(\delta
,\ve)}}\bigr)
\Phi^{(\delta
)}\bigl(Y^{(\delta,\ve)}\bigr)\bigr\|_{L^2([0,T];H_0^1(\mcO))}^2\le
M.\vspace*{1pt}
\]
By weak lower semicontinuity of the $L^\infty$ norm on $L^{m+1}$, the
convergence $Y^{(\delta,\ve)} \rightharpoonup^* Y$ in $L^\infty
([0,T];L^{m+1}(\mcO))$ and the convergence\vspace*{1pt}
\[
\Phi\bigl(e^{-\mu^{(\delta,\ve)}}\bigr)\Phi^{(\delta)}\bigl(Y^{(\delta,\ve)}\bigr)
\rightharpoonup\Phi \bigl(e^{-\mu^{(\ve)}}Y^{(\ve)}\bigr)\qquad \mbox{in }
L^2\bigl([0,T];H_0^1(\mcO)\bigr)\vspace*{1pt}
\]
obtained in the proof of Theorem \ref{thmexistenceweak}, these bounds continue to
hold for~$Y^{(\ve)}$.
\end{pf}

By Theorem \ref{thmexistenceweak} there is a weak solution $Y$ to
\eqref{eqnroughPDEtransformed} driven by $z$. Let $X:= e^{-\mu}Y$
and $X^{(\ve)}:= e^{-\mu^{(\ve)}}Y^{(\ve)}$. Then $X^{(\ve)}$ solves
\eqref{eqnroughPDE}, and we need to prove $X^{(\ve)}_t \to X_t$ in
$H$ for all $t \in[0,T]$. For this it is enough to prove $Y^{(\ve)}_t
\to Y_t$ in $H$ for all $t \in[0,T]$.

Lemma \ref{lemmaaprioriuniform} implies that $Y^{(\ve)}$ is
uniformly bounded in $L^\infty(\mcO_T)$, hence also in $L^\infty
([0,T];H)$. Moreover, $Z^{(\ve)} = \Phi(e^{-\mu^{(\ve)}}Y^{(\ve
)})$ is
uniformly bounded in $L^\infty(\mcO_T)$ and in $L^2([0,T];H_0^1(\mcO
))$. By the same argument as in Theorem \ref{thmexistenceweak} we
obtain the weak convergence $Y^{(\ve)}_t \rightharpoonup Y_t$ in $H$
for all $t \in[0,T]$ and $Z^{(\ve)} \rightharpoonup Z = \Phi(e^{-\mu}
Y)$ in $L^2([0,T];H_0^1(\mcO))$. Hence, $X^{(\ve)}_t \rightharpoonup
X_t:= e^{-\mu_t}Y_t$ in $H$ for all $t \in[0,T]$. Since $Y$ is the
unique weak solution to \eqref{eqnroughPDEtransformed} the uniform
bounds for $X$ follow from Theorem \ref{thmexistenceweak}.

It remains to prove that the convergence $X^{(\ve)}_t \rightharpoonup
X_t$ is strong in $H$. As in \eqref{eqnlimitIto} and \eqref
{eqnItoapprox}, we have
%
\begin{eqnarray}
\label{eqnlimitIto2} \|Y_t\|_H^2& = &
\|Y_s\|_H^2 - 2 \int_s^t
\int_{\mcO} e^{\mu_r} Z_r Y_r
\,d\xi \,dr
\nonumber
\\
&&{} + 2 \int_s^t \int_{\mcO}
Z_r \bigl( 2 \nabla\bigl(e^{\mu_r}\bigr) \nabla \bigl((-\D
)^{-1}(Y_r)\bigr)\\
&&\hspace*{75pt}{} + \D\bigl(e^{\mu_r}\bigr) (-
\D)^{-1}(Y_r) \bigr) \,d\xi \,dr\nonumber\
\end{eqnarray}
and
%
\begin{eqnarray}
\label{eqnItoapprox2} \bigl\|Y^{(\ve)}_t\bigr\|_H^2
&=& \bigl\|Y^{(\ve)}_s\bigr\|_H^2 - 2 \int
_s^t \int_{\mcO}
e^{\mu_r^{(\ve)}} Z_r^{(\ve)} Y_r^{(\ve)}
\,d\xi \,dr
\nonumber\\
&&{}+ 2 \int_s^t \int_{\mcO}
Z_r^{(\ve)} \bigl( 2 \nabla\bigl(e^{\mu
_r^{(\ve
)}}\bigr)
\nabla\bigl((-\D)^{-1}\bigl(Y_r^{(\ve)}\bigr)\bigr)
 \\
 &&\hspace*{82pt}{}+ \D\bigl(e^{\mu_r^{(\ve)}}\bigr) (-\D)^{-1}\bigl(Y_r^{(\ve)}
\bigr) \bigr) \,d\xi \,dr.\nonumber
\end{eqnarray}
Since $Y^{(\ve)} \in L^2([0,T];L^2(\mcO))$ and $\frac{dY^{(\ve)}}{dt}
\in L^2([0,T];H)$ are uniformly bounded and $L^2(\mcO) \hookrightarrow
\hookrightarrow H$, by the Aubin--Lions compactness theorem we have
\[
Y^{(\ve)} \to Y\qquad \mbox{strongly in } L^2\bigl([0,T];H\bigr).
\]
Integrating \eqref{eqnlimitIto2} and \eqref{eqnItoapprox2} over
$s \in[0,t]$ and subtracting yields
\begin{eqnarray*}
t \limsup_{\ve\to0} \bigl(\bigl\|Y^{(\ve)}_t
\bigr\|_H^2 - \|Y_t\|_H^2
\bigr) \le0,
\end{eqnarray*}
which implies strong convergence $Y^{(\ve)}_t \to Y_t$ in $H$.

%
%

\subsection{Limit solutions and dynamics on $L^1(\mcO)$}

\subsubsection{$L^1$-continuity and a comparison
principle}\label{ssecL1-ctn}
We will now prove uniform $L^1$ continuity in the initial condition
for weak solutions to \eqref{eqnroughPDEtransformed}. Using this
uniform continuity we can then construct limit solutions to \eqref
{eqnroughPDEtransformed}.

\begin{lemma}\label{lemmaswapdt}
Let $Y \in L^\infty([0,T];L^1(\mcO))$ such that $t \mapsto Y_t(\xi)$
is continuously differentiable on $[0,T]$ for almost all $\xi\in\mcO$
and $\partial_t Y \in L^1(\mcO_T)$. Then
\[
\int_\mcO Y_t^+ d\xi- \int_\mcO
Y_s^+ \,d\xi= \int_s^t\int
_\mcO \partial_r Y_r
\operatorname{sgn}^+(Y_r) \,d\xi \,dr,
\]
where $(\cdot)^+ = \operatorname{max}(\cdot,0)$ and $\operatorname{sgn}^+(\cdot) = \operatorname{max}(\operatorname{sgn}(\cdot),0)$.
\end{lemma}
\begin{pf}
Let $t \in[0,T]$. Since $Y \in L^\infty([0,T];L^1(\mcO))$, there is a
sequence $t_n \to t$ and a constant $M > 0$ such that $\|Y_{t_n}\|
_{L^1(\mcO)} \le M$. By continuity of $t \mapsto Y_t(\xi)$ for almost
all $\xi\in\mcO$ we have $Y_{t_n}(\xi) \to Y_t(\xi)$ almost
everywhere. Fatou's lemma yields $\|Y_t\|_{L^1(\mcO)} \le\liminf_{n
\to\infty} \|Y_{t_n}\|_{L^1(\mcO)} \le M$. Thus, $Y_t \in L^1(\mcO)$
for all $t \in[0,T]$. Let $\s^{(\tau)} \in C^\infty(\R)$ be such that
\[
\s^{(\tau)}(r):= \cases{ 0, & \quad$\mbox{for } r \le0,$ \vspace*{2pt}
\cr
r, &\quad
$\mbox{for } r \ge\tau,$}
\]
with $0 \le\dot\s^{(\tau)} \le1$ and $ 0 \le\ddot\s^{(\tau)}
\le
\frac{C}{\tau}$. For $0 \le s<t\le T$ we obtain
\begin{eqnarray*}
\int_\mcO\s^{(\tau)}(Y_t) \,d\xi- \int
_\mcO\s^{(\tau)}(Y_s) \,d\xi &=& \int
_s^t \int_\mcO
\partial_r \s^{(\tau)}(Y_r) \,d\xi \,dr
\\
&=& \int_s^t \int_\mcO
\dot\s^{(\tau)}(Y_r) \partial_r Y_r
\,d\xi \,dr.
\end{eqnarray*}
By dominated convergence this yields the assertion.
\end{pf}

\begin{lemma}\label{lemmaL1-ctnicweaksoln}
Let $Y_0^{i} \in L^\infty(\mcO)$, $i = 1,2$ and $Y^{(i)}$ be the
corresponding essentially bounded weak solution to \eqref
{eqnroughPDEtransformed}. Then there exists a constant $C > 0$ such that
\begin{eqnarray*}
&&\sup_{t \in[0,T]}\bigl \|\bigl(Y^{(1)}_t-Y^{(2)}_t
\bigr)^+\bigr\|_{L^1(\mcO)} + \bigl\|\bigl(\Phi \bigl(e^{-\mu}Y^{(1)}
\bigr) - \Phi\bigl(e^{-\mu}Y^{(2)}\bigr)\bigr)^+
\bigr\|_{L^1(\mcO_T)}
\\
&&\qquad\le C \bigl\|\bigl(Y^{(1)}_0-Y^{(2)}_0
\bigr)^+\bigr\|_{L^1(\mcO)}
\end{eqnarray*}
and
\begin{eqnarray*}
&&\sup_{t \in[0,T]} \bigl\|Y^{(1)}_t-Y^{(2)}_t
\bigr\|_{L^1(\mcO)} +\bigl \|\Phi \bigl(e^{-\mu}Y^{(1)}\bigr) - \Phi
\bigl(e^{-\mu}Y^{(2)}\bigr)\bigr\|_{L^1(\mcO_T)}
\\
&&\qquad\le C \bigl\|Y^{(1)}_0-Y^{(2)}_0
\bigr\|_{L^1(\mcO)}.
\end{eqnarray*}
\end{lemma}
\begin{pf}
Let $\s^{(\tau)}$ be as in the proof of Lemma \ref{lemmaswapdt}, and
let $\vp\in C^2(\bar\mcO)$ be the unique classical solution to
\begin{eqnarray*}
\D\vp&=& -1\qquad \mbox{in } \mcO,
\\
\vp&=& 1\qquad \mbox{on } \partial\mcO.
\end{eqnarray*}
By the maximum principle we have $\vp\ge1$. By Theorem \ref
{thmuniquenessveryweaksoln} the weak solutions $Y^{(i)}$ coincide
with the weak solutions constructed in the proof of Theorem \ref
{thmexistenceweak} by approximation with classical solutions
$Y^{(i,\delta
)}$ to \eqref{eqnapprox}. Let $z^{(\delta)} \in C^\infty([0,T];\R
^N)$ be
the corresponding smooth approximation of the driving signal $z$. By
equicontinuity of $z^{(\delta)}$ we can find a partition $0 = \tau_0
< \tau
_1 < \cdots < \tau_N =T$ of $[0,T]$ such that
\begin{eqnarray*}
&& \Bigl( \mathop{\inf_{\xi\in\bar\mcO, }}_{ t \in[\tau_i,\tau
_{i+1}]}
e^{\mu^{(\delta)}_t(\xi)-\mu^{(\delta)}_{\tau_i}(\xi)} \Bigr)\\
&&\quad{} \times \bigl( -1 + 2 \|\vp\|_{C^1(\mcO)} \bigl( \bigl\|
\nabla\bigl(\mu^{(\delta )}_t-\mu^{(\delta )}_{\tau_i}
\bigr)\bigr\|_{C^0(\mcO)} +
\bigl\|\nabla\bigl(\mu^{(\delta)}_t-\mu^{(\delta)}_{\tau_i}
\bigr)\bigr\| _{C^0(\mcO
)}^2 \\
&&\hspace*{194pt}\quad{}+ \bigl\|\D\bigl(\mu^{(\delta)}_t-
\mu^{(\delta)}_{\tau_i}\bigr)\bigr\|_{C^0(\mcO
)} \bigr) \bigr) \\
&&\qquad\le-
\frac{1}{2}
\end{eqnarray*}
for all $t \in[\tau_i,\tau_{i+1}]$, all $i = 0,\ldots,N-1$ and all
$\delta>
0$. Let now $\delta> 0$ be arbitrary, fixed. For simplicity we drop
the $\delta
$ dependency of the signal in the following calculation.
Define
\[
\eta_t(\xi):= \vp(\xi) \sum_{i=0}^{N-1}
\mathbh{1}_{[\tau
_i,\tau
_{i+1})}(t) e^{-\mu_{\tau_i}(\xi)}.
\]
For $\tau_i \le s < t < \tau_{i+1}$, by Lemma \ref{lemmaswapdt},
we have
\begin{eqnarray*}
&&\int_\mcO\bigl(Y^{(1,\delta)}_t-Y^{(2,\delta)}_t
\bigr)^+ \eta_t \,d\xi- \int_\mcO
\bigl(Y^{(1,\delta)}_s-Y^{(2,\delta)}_s\bigr)^+
\eta_s \,d\xi
\\
&&\qquad = \int_s^t\int_\mcO
\partial_r \bigl(Y^{(1,\delta
)}-Y^{(2,\delta)}\bigr) \operatorname{sgn}^+
\bigl(Y^{(1,\delta)}_r-Y^{(2,\delta)}_r \bigr)
\eta_r \,d\xi \,dr.
\end{eqnarray*}
Let $Y^{(\delta)}:= Y^{(1,\delta)}-Y^{(2,\delta)}$ and $w^{(\delta
)}=\Phi(e^{-\mu_r})
(\Phi^{(\delta)}(Y^{(1,\delta)})-\Phi^{(\delta)}(Y^{(2,\delta)}))$.
We observe
%
\begin{eqnarray}
\label{eqnL^1-ctn-1} &&\int_s^t
\int_\mcO\partial_r Y^{(\delta)} \operatorname{sgn}^+
\bigl(Y^{(\delta
)}_r \bigr) \vp e^{-\mu_{\tau_i}} \,d\xi \,dr
\nonumber
\\
&&\qquad = \int_s^t\int_\mcO \bigl(\D w^{(\delta)}_r\bigr)
\operatorname{sgn}^+\bigl(w^{(\delta)}_r \bigr) e^{\mu_r-\mu _{\tau_i}}
\vp \,d\xi \,dr
\nonumber \\[-8pt]
\\[-8pt]
\nonumber &&\qquad = \lim_{\tau\to0} \biggl( - \int_s^t \int_\mcO\nabla
w^{(\delta)}_r \nabla\bigl( \dot\s^{(\tau)}\bigl(w^{(\delta)}_r\bigr)
\bigr) e^{\mu_r-\mu_{\tau _i}} \vp\ \,d\xi \,dr
 \\
 &&\hspace*{59pt}{}- \int_s^t\int_\mcO
\nabla w^{(\delta)}_r \nabla\bigl( e^{\mu
_r-\mu_{\tau_i}} \vp\bigr)
\dot\s^{(\tau)}\bigl(w^{(\delta)}_r\bigr) \,d\xi \,dr \biggr).
\nonumber
\end{eqnarray}
Since $\nabla\dot\s^{(\tau)} (w^{(\delta)}_r) = \ddot\s^{(\tau
)}(w^{(\delta
)}_r) \nabla w^{(\delta)}_r$, the first term has negative sign. Partial
integration of the second term gives
\begin{eqnarray*}
&&- \int_\mcO\nabla w^{(\delta)} \nabla\bigl(
e^{\mu_r-\mu_{\tau_i}} \vp\bigr) \dot \s^{(\tau)}\bigl(w^{(\delta)}\bigr) \,d
\xi\\
&&\qquad= \int_\mcO w^{(\delta)} \D\bigl( e^{\mu_r-\mu_{\tau_i}}
\vp\bigr) \dot \s^{(\tau
)}\bigl( w^{(\delta)}\bigr) \,d\xi
 + \int_\mcO w^{(\delta)} \nabla\bigl(e^{\mu_r-\mu_{\tau
_i}}
\vp\bigr) \nabla\dot\s^{(\tau)}\bigl(w^{(\delta)}\bigr) \,d\xi.
\end{eqnarray*}
For the second term on the right-hand side, we note
\begin{eqnarray*}
&&\int_\mcO w^{(\delta)} \nabla\bigl(e^{\mu_r-\mu_{\tau_i}}
\vp\bigr) \nabla\dot\s ^{(\tau)}\bigl(w^{(\delta)}\bigr) \,d\xi
\\
&&\qquad= \int_{\mcO\cap\{ 0 < w^{(\delta)} < \tau\}} w^{(\delta)} \ddot \s^{(\tau
)}
\bigl(w^{(\delta)}\bigr) \nabla\bigl(e^{\mu_r-\mu_{\tau_i}} \vp\bigr) \cdot\nabla
w^{(\delta)} \,d\xi\to0
\end{eqnarray*}
for $\tau\to0$, by $\ddot\s^{(\tau)} \le\frac{C}{\tau}$ and
dominated convergence. Using dominated convergence we can take the
limit $\tau\to0$ in \eqref{eqnL^1-ctn-1} to get
\begin{eqnarray*}
\int_s^t\hspace*{-0.5pt}\int_\mcO
\partial_r Y^{(\delta)} \operatorname{sgn}^+\bigl(Y^{(\delta)}_r
\bigr) \vp e^{-\mu
_{\tau_i}} \,d\xi \,dr \le\int_s^t\hspace*{-0.5pt}
\int_\mcO w^{(\delta)}_r \D\bigl(
e^{\mu_r-\mu_{\tau_i}} \vp\bigr) \operatorname{sgn}^+ \bigl(w^{(\delta)}_r\bigr) \,d
\xi \,dr.
\end{eqnarray*}
We note
\begin{eqnarray*}
&&\D \bigl( e^{\mu_r-\mu_{\tau_i}} \vp \bigr)
\\
&&\qquad= e^{\mu_r-\mu_{\tau_i}} \bigl( \D\vp+ 2 \nabla\vp\cdot\nabla (\mu_r-
\mu_{\tau_i}) + \vp\bigl(\bigl|\nabla(\mu_r-\mu_{\tau_i})\bigr|^2+
\D (\mu_r-\mu _{\tau_i})\bigr) \bigr)
\\
&&\qquad\le- \frac{1}{2},
\end{eqnarray*}
by the choice of $\vp$ and $\tau_i$. Thus
\begin{eqnarray*}
\int_s^t\int_\mcO
\partial_r Y^{(\delta)} \operatorname{sgn}^+ \bigl(Y^{(\delta)}_r
\bigr) \vp e^{-\mu
_{\tau_i}} \,d\xi \,dr + \frac{1}{2} \int_s^t
\int_\mcO\bigl(w^{(\delta
)}_r\bigr)^+ \,d\xi
\,dr \le0.
\end{eqnarray*}
In conclusion,
\begin{eqnarray*}
&&\int_\mcO\bigl(Y^{(1,\delta)}_t-Y^{(2,\delta)}_t
\bigr)^+ \eta_t \,d\xi- \int_\mcO
\bigl(Y^{(1,\delta)}_s-Y^{(2,\delta)}_s\bigr)^+
\eta_s \,d\xi+ \frac{1}{2} \int_s^t
\int_\mcO\bigl(w^{(\delta)}_r\bigr)^+ \,d\xi
\,dr
\\
&&\qquad = \int_s^t\int_\mcO
\partial_r Y^{(\delta)} \operatorname{sgn}^+ \bigl(Y^{(\delta)}_r
\bigr) \eta_r \,d\xi \,dr + \frac{1}{2} \int_s^t
\int_\mcO\bigl(w^{(\delta
)}_r\bigr)^+ \,d\xi
\,dr \le0
\end{eqnarray*}
for all $\tau_i \le s < t < \tau_{i+1}$ and hence for all $0 \le s < t
\le T$. We have
\begin{eqnarray*}
& &\bigl\|\bigl(Y^{(1,\delta)}_t-Y^{(2,\delta)}_t
\bigr)^+\bigr\|_{L^1(\mcO)} + \bigl\|\Phi \bigl(e^{-\mu^{(\delta
)}}\bigr) \bigl(
\Phi^{(\delta)}\bigl(Y^{(1,\delta)}\bigr) - \Phi^{(\delta
)}
\bigl(Y^{(2,\delta)}\bigr)\bigr)^+\bigr\|_{L^1(\mcO
_T)}
\\
&&\qquad\le C \bigl\|\bigl(Y^{(1,\delta)}_0-Y^{(2,\delta)}_0
\bigr)^+\bigr\|_{L^1(\mcO)}
\end{eqnarray*}
for all $t \in[0,T]$, where the constant $C$ does not depend on
$\delta$
(using uniform boundedness of $z^{(\delta)}$). By the proof of Theorem
\ref
{thmexistenceweak} we know that $Y^{(i,\delta)}_t \rightharpoonup
Y^{(i)}_t$ in $L^1(\mcO)$ and $\Phi(e^{-\mu^{(\delta)}})\Phi
^{(\delta)}(Y^{(i,\delta
)}) \rightharpoonup\Phi(e^{-\mu}Y^{(i)})$ in $L^2([0,T];H_0^1(\mcO))$.
By weak lower semicontinuity of $\|(\cdot)^+\|_{L^1(\mcO)}$ and $\|
(\cdot)^+\|_{L^1(\mcO_T)}$, taking the limit $\delta\to0$ we obtain
\begin{eqnarray*}
&&\bigl\|\bigl(Y^{(1)}_t-Y^{(2)}_t\bigr)^+
\bigr\|_{L^1(\mcO)} + \bigl\|\bigl(\Phi\bigl(e^{-\mu}Y^{(1)}\bigr) -
\Phi\bigl(e^{-\mu}Y^{(2)}\bigr)\bigr)^+\bigr\|_{L^1(\mcO_T)}
\\
&&\qquad\le C \bigl\|\bigl(Y^{(1)}_0-Y^{(2)}_0
\bigr)^+\bigr\|_{L^1(\mcO)}.
\end{eqnarray*}
Since $Z^{(i)}:= -Y^{(i)}$ again is an essentially bounded weak
solution of \eqref{eqnroughPDEtransformed}, the same assertion
follows for $\|(Y^{(1)}_t-Y^{(2)}_t)^-\|_{L^1(\mcO)}$. Adding both
inequalities yields
\begin{eqnarray*}
&&\bigl\|Y^{(1)}_t-Y^{(2)}_t
\bigr\|_{L^1(\mcO)} +\bigl \|\Phi\bigl(e^{-\mu}Y^{(1)}\bigr) - \Phi
\bigl(e^{-\mu}Y^{(2)}\bigr)\bigr\|_{L^1(\mcO_T)} \\
&&\qquad\le C
\bigl\|Y^{(1)}_0-Y^{(2)}_0\bigr\|
_{L^1(\mcO)}.
\end{eqnarray*}
\upqed\end{pf}

\begin{remark}
Following the same argument, but with $\D\vp=-1$ with homogeneous
Dirichlet boundary conditions, the same result can be established in
the weighted $L^1$-space $L^1_\vp$. This then allows us to construct
limit solutions even for initial conditions in $L^1_\vp$.
\end{remark}

Using this uniform $L^1$ continuity in the initial condition, we can
now construct limit solutions for all initial conditions in $L^1$.

\begin{pf*}{Proof of Theorem \ref{thmlimitsoln}}
Let $Y_0:= e^{\mu_0}X_0 \in L^1(\mcO)$ and $Y^{(\delta)}_0 \to Y_0$ in
$L^1(\mcO)$ with $Y^{(\delta)}_0 \in L^\infty(\mcO)$. Let
$Y^{(\delta)}$ be the
essentially bounded weak solution corresponding to $Y^{(\delta)}_0$. By
Lemma \ref{lemmaL1-ctnicweaksoln} we have
\begin{eqnarray*}
&&\sup_{t \in[0,T]}\bigl\|Y^{(\delta_1)}_t-Y^{(\delta_2)}_t
\bigr\|_{L^1(\mcO
)} + \bigl\|\Phi \bigl(e^{-\mu}Y^{(\delta_1)}\bigr) - \Phi
\bigl(e^{-\mu}Y^{(\delta_2)}\bigr)\bigr\| _{L^1(\mcO_T)}
\\
&&\qquad\le C \bigl\|Y^{(\delta_1)}_0-Y^{(\delta_2)}_0
\bigr\|_{L^1(\mcO)}
\end{eqnarray*}
for all $\delta_1,\delta_2 > 0$. Hence, $Y^{(\delta)}_t$ is a Cauchy
sequence in
$L^1(\mcO)$ and thus uniformly convergent to some limit $Y_t \in
L^1(\mcO)$. Since $\Phi(e^{-\mu}Y^{(\delta_1)})$ is a Cauchy
sequence in
$L^1(\mcO_T)$, and $\Phi$ is continuous, we obtain $\Phi(e^{-\mu
}Y^{(\delta
_1)}) \to\Phi(e^{-\mu}Y)$ in $L^1(\mcO_T)$.

By Theorem \ref{thmroughPDE}, $X^{(\delta)} = e^{-\mu} Y^{(\delta
)}$ are
rough weak solutions, and we conclude $X^{(\delta)}_t \to X_t:=
e^{-\mu_t}
Y_t$ uniformly in $L^1(\mcO)$ and $\Phi(X^{(\delta)}) \to\Phi(X)$ in
$L^1(\mcO_T)$. In the proof of Theorem \ref{thmexistenceweak}, we
have proven weak continuity of $t \mapsto Y_t^{(\delta)}$ in $L^p(\mcO)$.
Hence $t \mapsto X^{(\delta)}_t$ is weakly continuous in $L^1(\mcO)$ and
thus is $t \mapsto X_t$. The bound $X_t \le U_t$ follows immediately.
\end{pf*}

\subsection{Equicontinuity of solutions}\label{sseccontinuity}
\mbox{}
\begin{pf*}{Proof of Theorem \ref{thmctnlimitsoln}}
We only prove (i). The proofs of (ii) and (iii) are analogous. Let
$X_0 \in L^1(\mcO)$ and $X$ be the corresponding limit solution. Since
$K \subseteq(0,T] \times\mcO$ is compact, there is a $\tau>0$ such
that $K \subseteq[\tau,T] \times\mcO$. By Theorem~\ref
{thmlimitsoln} we know that $Y = e^{\mu}X \in L^\infty([\tau,T]
\times\mcO)$, and by Remark \ref{rmklimitareveryweak} $Y$ is a
very weak solution of \eqref{eqnroughPDEtransformed}. By Theorems
\ref
{thmuniquenessveryweaksoln} and~\ref{thmexistenceweak}
this implies that $Y$ is an essentially bounded weak solution to \eqref
{eqnroughPDEtransformed} on $[\tau,T]\times\mcO$ with initial
condition $Y_\tau$. Due to the uniform $L^\infty$ bound $U$ established
in Theorem \ref{thmlimitsoln}, $\|Y_\tau\|_{L^\infty(\mcO)}$ is
bounded independent of the initial condition $Y_0$. It is thus
sufficient to prove the claimed regularity for weak solutions $Y$ of
\eqref{eqnroughPDEtransformed} with a modulus of continuity
depending only on the data and $\|Y_0\|_{L^\infty(\mcO)}$.

Let $Y^{(\delta)}$ be the sequence of approximating solutions with initial
condition $Y^{(\delta)}_0$ and driving signal $z^{(\delta)}$ used in Theorem
\ref{thmexistenceweak}. By Theorem \ref{thmexistenceweak} and Lemma
\ref{lemmaL-infty-boundclassicalsolns}, $Y$ and $Y^{(\delta)}$ are
uniformly bounded, that is,
\[
\|Y\|_{L^\infty(\mcO_T)}, \bigl\|Y^{(\delta)}\bigr\|_{L^\infty(\mcO_T)} \le M \qquad\mbox{for all
} \delta\le\delta_0
\]
for some constant $M > 0$ depending on $\|Y_0\|_{L^\infty(\mcO)}$. We
aim to apply the continuity results for porous media type PDE given in
\cite{DB83} to the approximating equation~\eqref{eqnapprox}. In
\cite
{DB83} equations of the form
%
\begin{eqnarray}
\label{eqnDB} \frac{d}{dt} \b(v) &= \div a(t,\xi,v,\nabla v) + b(t,\xi,v,
\nabla v)\qquad \mbox{on } \mcO_T
\end{eqnarray}
with homogeneous Dirichlet boundary conditions and initial value $v_0$
are considered. We first rewrite the approximating equations in the
form of \eqref{eqnDB}. The approximating equation \eqref{eqnapprox}
(driven by $z^{(\delta)}$) is equivalent to
\begin{eqnarray*}
\partial_t Y^{(\delta)}_t &=& \div a^{(\delta)}
\bigl(t,\xi,\Phi ^{(\delta)}\bigl(Y^{(\delta
)}_t\bigr),\nabla
\Phi^{(\delta)}\bigl(Y^{(\delta)}_t\bigr)\bigr) \\
&&{}+
b^{(\delta)}\bigl(t,\xi,\Phi^{(\delta
)}\bigl(Y^{(\delta)}_t
\bigr),\nabla\Phi^{(\delta)}\bigl(Y^{(\delta)}_t\bigr)\bigr)
\end{eqnarray*}
with
\begin{eqnarray*}
a^{(\delta)}(t,\xi,z,p) &=& e^{(1-m) \mu^{(\delta)}_t(\xi)} p,
\\
b^{(\delta)}(t,\xi,z,p) &=& e^{\mu^{(\delta)}_t(\xi)} \D\bigl(\Phi
\bigl(e^{-\mu^{(\delta
)}_t(\xi)}\bigr)\bigr) z - (m+1) e^{(1-m) \mu^{(\delta)}_t(\xi)} \nabla\mu
^{(\delta
)}_t(\xi) \cdot p.
\end{eqnarray*}
Let $\b^{(\delta)}:=  ( \Phi^{\delta}  )^{-1}$. For the
approximating
solutions $Y^{(\delta)}$ we define $Z^{(\delta)}:= \Phi^{\delta
}(Y^{(\delta)})$. Then
$Z^{(\delta)}$ satisfies
%
\begin{equation}
\label{eqnDB-1} \partial_t \b^{(\delta)}\bigl(Z^{(\delta)}_t
\bigr) = \div a^{(\delta
)}\bigl(t,\xi,Z^{(\delta
)}_t,\nabla
Z^{(\delta)}_t\bigr) + b^{(\delta)}\bigl(t,
\xi,Z^{(\delta)}_t, \nabla Z^{(\delta)}_t\bigr).
\end{equation}

The continuity of solutions to equations of this type has been shown
in \cite{DB83} under the assumption of an a priori $L^\infty
([0,T]\times
\mcO)$-bound and a growth bound for~$b$ (among other assumptions). The
growth bound on $b$ used in \cite{DB83} is not satisfied by \eqref
{eqnDB-1}. However, using the a priori $L^\infty$ bound on
$Y^{(\delta)}$,
we can cut-off $b$ in the $z$ variable without changing the solution
property of $Y^{(\delta)}$, thus guaranteeing that the growth
condition is
satisfied. We modify $\Phi^{(\delta)}$ on $\R\setminus[-M,M]$ to obtain
$\dot\Phi^{(\delta,M)} \le C_2$ uniformly in $\delta$ [while preserving
properties (i)--(iii) in \eqref{eqnapproxprop}], and we modify $b$ by
\begin{eqnarray*}
&&b^{(M,\delta)}(t,\xi,z,p) \\
&&\qquad= e^{\mu^{(\delta)}_t(\xi)} \D\bigl(\Phi
\bigl(e^{-\mu^{(\delta
)}_t(\xi)}\bigr)\bigr) z \mathbh{1}_{|z| \le M} - (m+1)
e^{(1-m) \mu^{(\delta
)}_t(\xi
)} \nabla\mu^{(\delta)}_t(\xi) \cdot p.
\end{eqnarray*}
Let $\b^{(\delta,M)}:=  ( \Phi^{(\delta,M)}  )^{-1}$.
Using the
$L^\infty$ bound we realize that $Z^{(\delta)}$ is a solution of
\begin{eqnarray*}
\label{eqncuttedapproxeqn} \partial_t \b^{(\delta,M)}
\bigl(Z^{(\delta)}_t\bigr) & =& \div a^{(\delta
)}\bigl(t,
\xi,Z^{(\delta
)}_t,\nabla Z^{(\delta)}_t\bigr) +
b^{(M,\delta)}\bigl(t,\xi,Z^{(\delta)}_t, \nabla
Z^{(\delta
)}_t\bigr),
\\
Z^{(\delta)}(0) &=& Z^{(\delta)}_0:= \Phi^{(\delta)}
\bigl(Y^{(\delta
)}_0\bigr)\qquad \mbox{on } \mcO
\end{eqnarray*}
for $M$ large enough. By \cite{DB83}, we obtain that $Z^{(\delta)}$ and
thus $Y^{(\delta)}$ are equicontinuous on $K$ with modulus of continuity
depending only on the data, $\|Y_0\|_{L^\infty(\mcO)}$ and
$\operatorname{dist}(K,\partial\mcO_T)$. Hence, the set $\{Y^{(\delta)}| \delta>
0\}$ is a
compact subset of $C(K)$, and we can choose a uniformly convergent
subsequence. By the proof of existence of weak solutions we know that
$Y^{(\delta)} \rightharpoonup Y$ in $L^{m+1}(\mcO_T)$. Consequently
$Y^{(\delta
)} \to Y$ uniformly on $K$. This implies $Y \in C(K)$ with the same
modulus of continuity.
\end{pf*}

\begin{pf*}{Proof of Corollary \ref{corL^1-timectnlimitsoln}}
Let $X_0 \in L^1(\mcO)$ and $X$ be the corresponding limit solution.
By Theorem \ref{thmctnlimitsoln}, $t \mapsto X_t(\xi)$ is continuous
on $(0,T]$ for each $\xi\in\mcO$. By Theorem \ref{thmlimitsoln} $X$
is uniformly bounded on $[\tau,T]\times\mcO$ for all $\tau> 0$.
Dominated convergence implies $X \in C((0,T];L^p(\mcO))$. We can
approximate $X_0$ by $X_0^{(\delta)} \in C(\bar\mcO)$ such that
$X_0^{(\delta)}
\to X_0$ in $L^1(\mcO)$. Let $X^{(\delta)}$ be the weak solution
corresponding to $X_0^{(\delta)}$. By Theorem \ref{thmctnlimitsoln}(ii), $t \mapsto X^{(\delta)}_t(\xi)$ is continuous on $[0,T]$ for each
$\xi\in\mcO$ and by Theorem \ref{thmexistenceweak} $X^{(\delta
)}$ is
uniformly bounded in $[0,T]\times\mcO$. Dominated convergence implies
$X^{(\delta)} \in C([0,T];L^1(\mcO))$. By Theorem \ref
{thmlimitsoln} we
have $\sup_{t \in[0,T]} \|X^{(\delta)}_t - X_t\|_{L^1(\mcO)} \to
0$, hence
also $X \in C([0,T];L^1(\mcO))$. If $X_0 \in L^\infty(\mcO)$, then by
uniqueness of essentially bounded weak solutions and Theorem~\ref
{thmexistenceweak}, $X$ is uniformly bounded in $[0,T]\times\mcO$.
Since also $X \in C([0,T];L^1(\mcO))$ this implies $X \in
C([0,T];L^p(\mcO))$ by dominated convergence.
\end{pf*}

%
%
\section{Generation of an RDS and random attractors}\label{secRDS}
%
\subsection{Transformation in the semimartingale case}
\mbox{}
\begin{pf*}{Proof of Theorem \ref{thmtransformationveryweaksemimartingale}}
Let $z$ be a continuous semimartingale in $\R^N$, $X$~be the limit
solution to \eqref{eqnroughPDE} and $Y:= e^{\mu}X$. By Remark \ref
{rmklimitareveryweak}, $Y$ is a very weak solution to \eqref
{eqnroughPDEtransformed}. We will now prove that $X$ satisfies
\eqref
{eqnroughPDEveryweak}.

We consider the Sobolev spaces $H_0^{2k}(\mcO)$ with the norm $\|\cdot
\|_{H_0^{2k}(\mcO)}:= \|(-\D)^k\cdot\|_2$. By Sobolev embeddings there
is a $k \in\N$ (w.l.o.g. $k$ odd) such that $H_0^{2k}(\mcO)
\hookrightarrow C^0(\mcO)$ continuously. Hence $L^1(\mcO)
\hookrightarrow(H_0^{2k}(\mcO))^* =: H^{-{2k}}$ and $Y \in
C([0,T];L^1(\mcO)) \subseteq C([0,T];H^{-{2k}})$. Let $\vp\in
H_0^{2(k+1)}(\mcO)$ and $\td e_j$ be an orthonormal basis of $H_0^{2k}$
given by $\td e_j = \frac{e_j}{\l_j^k} = (-\D)^{-k}e_j$, where $e_j$ is
an orthonormal basis of eigenvectors of $-\D$ on $L^2(\mcO)$ with
homogeneous Dirichlet boundary conditions and $\l_k$ are the
corresponding eigenvalues. Further, let $P_M\dvtx H^{-2k} \to \operatorname{span}\{e_1,
\ldots, e_M\}$ be the orthogonal projection. Then ${P_M}_{|L^2(\mcO)}$,
${P_M}_{|H^{2k}_0(\mcO)}$ are the orthogonal projections onto $\operatorname{span}\{
e_1, \ldots, e_M\}$ in $L^2(\mcO)$, $H^{2k}_0(\mcO)$, respectively. We have
\begin{eqnarray*}
\int_\mcO X_t \vp \,d\xi &=& _{H^{-{2k}}}\bigl
\langle Y_t,e^{-\mu_t} \vp\bigr\rangle_{H_0^{2k}} \\
&=& \sum
_{j=1}^\infty \biggl( \int
_\mcO Y_t e_j \,d\xi \biggr)
\bigl(e_j,e^{-\mu
_t} \vp\bigr)_2.
\end{eqnarray*}
By the very weak solution property and continuity in $L^1(\mcO)$,
\[
\int_\mcO Y_t e_j \,d\xi= \int
_\mcO Y_s e_j \,d\xi+ \int
_s^t \int_\mcO \Phi
\bigl(e^{-\mu_r}Y_r\bigr) \D\bigl(e^{\mu_r}e_j
\bigr) \,d\xi \,dr\qquad \forall s \le t,
\]
and hence $t \mapsto\int_\mcO Y_t e_j \,d\xi$ is an absolutely
continuous map with derivative $\int_\mcO\Phi(e^{-\mu_t}Y_t) \D
(e^{\mu
_t}e_j)  \,d\xi$ for a.e. $t \in[0,T]$. By Theorem \ref
{thmroughPDE}, $Y_t$ is adapted. As in \cite{BDPR09-2}, page 22, by
use of the stochastic Fubini theorem (cf., e.g., \cite{vNV06}), we prove
\[
\bigl(e_j,e^{-\mu_t}\vp\bigr)_2 =
\bigl(e_j,e^{-\mu_0}\vp\bigr)_2 + \sum
_{k=1}^N \int_0^t
\bigl(e_j,f_k e^{-\mu_r}\vp\bigr)_2
\circ dz^{(k)}_r.
\]
In particular $(e_j,e^{-\mu_t}\vp)_2$ is a real-valued continuous
semimartingale. Hence, we can apply the It\^o product rule (cf.
\cite{P04}, page~83) to get
%
\begin{eqnarray}
\label{eqnsemimarttransf1} && \biggl( \int_\mcO
Y_t e_j \,d\xi \biggr) \bigl(e_j,e^{-\mu_t}
\vp\bigr)_2
\nonumber\\
&&\qquad= \biggl( \int_\mcO Y_s e_j \,d
\xi \biggr) \bigl(e_j,e^{-\mu_s}\vp\bigr)_2
\nonumber
\\[-8pt]
\\[-8pt]
\nonumber
&&\quad\qquad{} + \int_s^t \bigl(e_j,e^{-\mu_{r}}
\vp\bigr)_2 \biggl( \int_\mcO \Phi
(X_r) \D\bigl(e^{\mu_r} \td e_j\bigr) \,d\xi
\biggr) \,dr
\\
&&\quad\qquad{} + \sum_{k=1}^N \int
_s^t \biggl( \int_\mcO
Y_r e_j \,d\xi \biggr) \bigl(e_j,f_k
e^{-\mu_r}\vp\bigr)_2 \circ \,dz^{(k)}_r\nonumber
\end{eqnarray}
for all $0 \le s \le t \le T$, $\P$-almost surely. Note
\[
\int_\mcO\Phi(X_r) \D\bigl(e^{\mu_r}
\td e_j\bigr) \,d\xi = \int_\mcO
\Phi(X_r) \bigl( \td e_j \D e^{\mu_r} + 2 \nabla
e^{\mu_r} \cdot\nabla\td e_j + e^{\mu_r} \D\td
e_j\bigr) \,d\xi.
\]
We aim to sum over $j$ in \eqref{eqnsemimarttransf1}. For this we
have to rewrite the second summand on the right-hand side of the
equation above. Due to the lack of regularity of $\Phi(X)$ this
requires an additional approximation,
\begin{eqnarray*}
&&\int_\mcO\Phi(X_r) \nabla e^{\mu_r}
\cdot\nabla\td e_j \,d\xi
\\
&&\qquad= - \lim_{M \to\infty} \int_\mcO \bigl( \nabla
P_M \Phi(X_r) \cdot \nabla e^{\mu_r} +
P_M \Phi(X_r) \D\bigl( e^{\mu_r}\bigr) \bigr)
\td e_j \,d\xi.
\end{eqnarray*}
Hence,
\begin{eqnarray*}
&&\sum_{j=1}^K \bigl(\td
e_j,e^{-\mu_{r}}\vp\bigr)_{H_0^{2k}} \int
_\mcO\Phi(X_r) 2 \nabla e^{\mu_r} \cdot
\nabla\td e_j \,d\xi
\\
& &\qquad= - 2 \lim_{M \to\infty} \biggl( \int_\mcO
\bigl( \nabla P_M \Phi(X_r) \cdot\nabla
e^{\mu_r} \bigr)P_K\bigl(e^{-\mu_{r}}\vp\bigr) \,d\xi
\\
&&\hspace*{80pt}{}+\int_\mcO \bigl( P_M \Phi(X_r)
\D\bigl( e^{\mu_r}\bigr) \bigr) P_K\bigl(e^{-\mu_{r}}\vp
\bigr) \,d\xi \biggr).
\end{eqnarray*}
We obtain
\begin{eqnarray*}
&& \sum_{j=1}^\infty\bigl(\td
e_j,e^{-\mu_{r}}\vp\bigr)_{H_0^{2k}} \int
_\mcO \Phi (X_r) \D\bigl(e^{\mu_r} \td
e_j\bigr) \,d\xi
\\
&&\qquad = {}_{H^{-{2k}}}\!\bigl\langle\Phi(X_r) \D e^{\mu_r},
e^{-\mu_{r}}\vp \bigr\rangle_{H_0^{2k}} + {}_{H^{-{2k}}}\!\bigl\langle
\Phi(X_r) e^{\mu_r}, \D \bigl(e^{-\mu _{r}}\vp\bigr) \bigr
\rangle_{H_0^{2k}}
\\
&&\quad\qquad{} + 2 _{H^{-{2k}}}\bigl\langle\Phi(X_r), \nabla
e^{\mu_r} \cdot \nabla\bigl(e^{-\mu_{r}}\vp\bigr)\bigr
\rangle_{H_0^{2k}}
\\
&&\qquad = {}_{H^{-{2k}}}\!\bigl\langle\Phi(X_r), \D\vp\bigr
\rangle_{H_0^{2k}}.
\end{eqnarray*}
Summing up $j=1,\ldots,\infty$ in \eqref{eqnsemimarttransf1} yields
\begin{eqnarray*}
&&\int_\mcO X_t \vp \,d\xi
\\
&&\qquad= \int_\mcO X_s \vp \,d\xi+ \int
_s^t \int_\mcO
\Phi(X_r) \D\vp \,d\xi \,dr + \int_s^t
\biggl( \int_\mcO B(X_r)\vp \,d\xi \biggr) \circ
\,dz_r
\end{eqnarray*}
for all $0\le s \le t \le T$ and all $\vp\in H^{2(k+1)}_0(\mcO)$ [thus
by approximation for all $\vp\in C^2_0(\bar\mcO)$] $\P$-almost surely.
\end{pf*}

\subsection{Quasi-continuous random dynamical systems}
\mbox{}
\begin{pf*}{Proof of Lemma \ref{lemmaascmpolimit}}
Since $\tau$ is weaker than the norm topology, we have $\Omega
(B,\omega)
\subseteq\Omega^\tau(B,\omega)$. Let now $y \in\Omega^\tau
(B,\omega)$. Then
there are $t_n \to\infty$ and $x_n \in B(\t_{-t_n}\omega)$ such
that $\vp
(t_n,\t_{-t_n}\omega)x_n \to^\tau y$. By $\mcD$ asymptotic compactness
there is a convergent subsequence $\vp(t_{n_k},\t_{-t_{n_k}}\omega
)x_{n_k}$. Since $\tau$ is weaker than norm topology and Hausdorff, we
conclude $\vp(t_{n_k},\t_{-t_{n_k}}\omega)x_{n_k} \to y \ni\Omega
(B,\omega) $.
\end{pf*}

\begin{pf*}{Proof of Lemma \ref{lemmaascmpolimit2}}
Without loss of generality we may assume that $F$ is a bounded $\mcD
$-absorbing set by augmenting $F$ to some $\ve$-neighborhood of $F$ for
some $\ve> 0$.

Let $t_n \to\infty$ and $x_n \in B(\t_{-t_n}\omega)$. Then there is a
convergent subsequence $\vp(t_{n_l},\t_{t_{n_l}}\omega)x_{t_{n_l}}
\to x
\in\Omega(B,\omega)$. Hence, $\Omega(B,\omega)$ is nonempty.

\textit{Compactness}:
Let $x_n \in\Omega(B,\omega)$. For every $n \in\N$ there are sequences
$t_{k(n)} \to\infty$ and $y_{k(n)} \in B(\t_{-t_{k(n)}}\omega)$
such that
$\vp(t_{k(n)},\t_{-t_{k(n)}}\omega)y_{k(n)} \to x_n$ for $k(n) \to
\infty$.
Therefore, we can find sequences $t_n \to\infty$, $y_n \in B(\t
_{-t_n}\omega)$ such that $\|\vp(t_n,\t_{-t_n}\omega)y_n -x_n\|_X <
\frac
{1}{n}$. By $\mcD$-asymptotic compactness there is a convergent
subsequence $\vp(t_{n_l},\t_{-t_{n_l}}\omega)y_{n_l} \to x \ni
\Omega(B,\omega
)$. Hence, $x_{n_l} \to x \ni\Omega(B,\omega)$.

\textit{Invariance}:  First let $x \in\Omega(B,\omega)$. We need to prove
$\vp(t,\omega)x \in\Omega(B,\t_t\omega)$. Since $x \in\Omega
(B,\omega)$ there are
sequences $t_n \to\infty$, $x_n \in B(\t_{-t_n}\omega)$ such that
$\vp
(t_n,\t_{-t_n}\omega)x_n \to x$. By the cocycle property $\vp
(t+t_n,\t
_{-t_n}\omega)x_n=\break  \vp(t,\omega)\vp(t_n,\t_{-t_n}\omega)x_n$ and
by bounded
absorption $\vp(t+t_n,\t_{-t_n}\omega)x_n = \vp(t+t_n,\t
_{-(t+t_n)}\t_t\omega
)x_n \in F(\t_t\omega)$ for $n$ large enough.
By quasi-$\tau$-continuity we conclude $\vp(t+t_n,\t_{-t_n}\omega
)x_n \to
^\tau\vp(t,\omega)x$. Hence $\vp(t,\omega)x \in\Omega^\tau(B,\t
_t\omega) =  \Omega
(B,\t_t\omega)$.

Let now $z \in\Omega(B,\t_t\omega)$, that is,
%
\begin{equation}
\label{eqninv1} \vp(t_n,\t_{-t_n}\t_t
\omega)x_n \to z
\end{equation}
for some $t_n \to\infty$ and $x_n \in B(\t_{-t_n}\t_t\omega)$. By
$\mcD
$-asymptotic compactness of $\vp$ there is a subsequence $\vp
(t_{n_l}-t,\t_{-(t_{n_l}-t)}\omega)x_{n_l} \to x \ni\Omega(B,\omega
)$. By
\eqref{eqninv1}, quasi-$\tau$-continuity and the cocycle property, we
have $\vp(t_{n_l},\t_{-t_{n_l}}\t_t\omega)x_{n_l} = \vp(t,\omega)
\vp
(t_{n_l}-t,\t_{-(t_{n_l}-t)}\omega)x_{n_l} \to^\tau\vp(t,\omega
)x$. Since $\tau
$ is weaker than norm topology and Hausdorff, we conclude $z = \vp
(t,\omega
)x$ with $x \in\Omega(B,\omega)$.
\end{pf*}

\begin{pf*}{Proof of Theorem \ref{thmexra}}
Necessity of the conditions follows from compactness of $A$ and its
attraction property. To prove sufficiency we first observe that by
Lemma \ref{lemmaascmpolimit2},
\[
A(\omega):= \O(F,\omega)
\]
is compact and invariant. Since $F \in\mcD$ and $F$ is $\mcD
$-attracting we have $A(\omega) \subseteq F(\omega)$ for all $\omega
\in\O$ and
thus $A \in\mcD$. We only need to prove attraction. We first observe that
\[
\O(D,\omega) \subseteq\O(F,\omega) = A(\omega) \qquad\forall D \in\mcD, \omega\in \O.
\]
Indeed, by attraction we have $\O(B,\omega) \subseteq F(\omega)$. By
Lemma \ref
{lemmaascmpolimit2} we know that $\O(B,\omega) = \vp(t,\t
_{-t}\omega)\O
(B,\t_{-t}\omega) \subseteq\vp(t,\t_{-t}\omega)F(\t_{-t}\omega
)$. Hence
\[
\O(B,\omega) \subseteq\bigcap_{t \ge0} \vp(t,
\t_{-t}\omega)F(\t _{-t}\omega) \subseteq\O(F,\omega) = A(
\omega).
\]
Assume that $A$ is not attracting. Then there is a set $B \in\mcD$, an
$\omega\in\O$, sequences $t_n \to\infty$, $x_n \in B(\t
_{-t_n}\omega)$ and a
$\delta> 0$ such that
\[
d\bigl(\vp(t_n,\t_{-t_n}\omega)x_n, A(
\omega)\bigr) \ge\delta
\]
for all $n \in\N$. By asymptotic compactness, there is a convergent
subsequence $\vp(t_{n_l},\t_{-t_{n_l}}\omega)x_{n_l} \to x \in
\Omega(B,\omega)
\subseteq A(\omega)$, which implies a contradiction.
\end{pf*}

\subsection{\texorpdfstring{Construction of an RDS for \protect\eqref{eqnroughPDE}}
{Construction of an RDS for (1.2)}}
\mbox{}
\begin{pf*}{Proof of Theorem \ref{thmgenerationRDS}}
By Theorem \ref{thmlimitsoln} the map $x \mapsto X(t,s;\omega)x$ is
Lipschitz continuous in $X=L^1(\mcO)$, locally uniformly in $s,t$.
Uniqueness of essentially bounded very weak solutions implies the flow property
\[
X(t,s;\omega)x = X(t,r;\omega)X(r,s;\omega)x\qquad\forall\omega \in\O, s \le r \le t
\]
and cocycle property
\[
X(t,s;\t_r \omega)x = X(t+r,s+r;\omega)x\qquad \forall\omega\in \O, s
\le t, r \in\R
\]
for all $x \in L^\infty(\mcO)$. By Lipschitz continuity in the initial
condition, these properties remain true for all $x \in X = L^1(\mcO)$.

Next, we prove measurability of the map $(t,s,\omega,x) \mapsto
X(t,s;\omega
)x$. First let $t \ge s$ and $x \in L^\infty(\mcO)$. By Theorem \ref
{thmroughPDE} the map $\mu\to X_t(\mu)$ from $C(\R;\R^N)$ to $H$ is
continuous. Since also $\omega\mapsto\mu(\omega)$ is a measurable
map, this
implies measurability of $\omega\mapsto X(t,s;\omega)x$ in $H$. Hence
\[
\omega\mapsto\int_\mcO \bigl( X(t,s;\omega)x \bigr) h \,d\xi
\]
is measurable for all $h \in H_0^1(\mcO)$. Since $X=L^1(\mcO)$ is
separable, by the Pettis measurability theorem, this implies
measurability of $X(t,s;\cdot)x$. By approximation, this remains true
for all $x \in X$. Since $X(\cdot,s;\omega)x \in C([0,T];X)$, for all
$\omega
\in\O$, we deduce joint measurability of $(t,\omega) \mapsto
X(t,s;\omega)x$
in $X$. Using $X(t,s;\omega)x = X(t-s,0;\t_s\omega)x$ and joint measurability
of $(s,\omega) \mapsto(t-s,\t_s\omega)$ this implies measurability
of $(s,\omega)
\mapsto X(t,s;\omega)x$. Hence, measurability of $(t,s,\omega,x)
\mapsto
X(t,s;\omega)x$ follows, and $\vp$ defines a continuous RDS on
$L^1(\mcO)$.

By Theorem \ref{thmlimitsoln}, $\vp(t,\omega)x \in L^p(\mcO)$ for
all $t
\in\R_+$ if $x \in L^p(\mcO)$, $p \in[1,\infty]$. Since $L^p(\mcO)$
is reflexive for $p \in(1,\infty)$ this implies quasi-weak-continuity
of $\vp$ on $L^p(\mcO)$ for all $p \in(1,\infty)$ by Proposition
\ref
{propquasi-tau-cont}. For $p = \infty$ we note that $\s(L^\infty
,\overline{i^*(L^\infty)})$ is the weak$^*$ topology. By Proposition
\ref{propquasi-tau-cont} quasi-weak$^*$-continuity of $\vp$ on
$L^\infty(\mcO)$ follows.
\end{pf*}

\subsection{\texorpdfstring{Bounded absorption, asymptotic compactness and random attractors for~$\vp$}
{Bounded absorption, asymptotic compactness and random attractors for phi}}

In the following let $\mcD$ be the universe of all random closed sets.

\begin{proposition}[(Bounded absorption)]\label{propbddabs}
There is an $L^\infty(\mcO)$-bounded (i.e., $\|F(\omega)\|_{L^\infty
(\mcO
)} < \infty$) $\mcD$-absorbing random set $F \in\mcD$. The absorption
time for $D \in\mcD$, $\omega\in\O$ can be chosen independent of
$\omega$
and $D$.
\end{proposition}
\begin{pf}
Recall that by Theorem \ref{thmgenerationRDS} we have $\vp(t,\omega)x
\le U_t(\omega)$ a.e. in $\mcO$ for all $t \ge0$ and all $x \in X$. For
$D \in\mcD$:
\begin{eqnarray*}
\vp(t,\t_{-t}\omega)D(\t_{-t}\omega) = \vp(1,
\t_{-1}\omega) \vp (t-1,\t_{-t}\omega)D(\t _{-t}
\omega) \le U_{1}(\t_{-1}\omega),
\end{eqnarray*}
a.e. in $\mcO$ for all $t \ge1$. Hence,
\[
F(\omega) = \bigl\{x \in L^\infty(\mcO)| \|x\|_{L^\infty(\mcO)} \le\bigl\|
U_{1}(\t _{-1}\omega)\bigr\|_{ L^\infty(\mcO)}\bigr\}
\]
is a $\mcD$-absorbing set with absorption time $t \equiv1$.
\end{pf}

\begin{lemma}[(Asymptotic compactness)]
\begin{longlist}[(ii)]
\item[(i)] The RDS $\vp$ is $\mcD$-asymptotically compact on each
$L^p(\mcO)$, $p \in[1,\infty)$.
\item[(ii)] If $(\mcO1)$ is satisfied, then there exists a compact
$\mcD$-absorbing set $K$ with $K(\omega) \subseteq C^0(\bar\mcO)$ compact
for each $\omega\in\O$. In particular, $\vp$ is $\mcD$-asymptoti\-cally
compact on $L^\infty(\mcO)$.
\end{longlist}
\end{lemma}
\begin{pf}
(i): Let $t_n \to\infty$, $D \in\mcD$ and $x_n \in D(\t
_{-t_n}\omega)$.
In Proposition \ref{propbddabs} we have proved the existence of a
$\mcD$-absorbing random set $F$. Note
\begin{eqnarray*}
\vp(t_n,\t_{-t_n}\omega)x_n &=& \vp(1,
\t_{-1}\omega) \vp(t_n-1,\t_{-(t_n-1)}
\t_{-1}\omega)x_n
\\
&\subseteq&\vp(1,\t_{-1}\omega) F(\omega)
\end{eqnarray*}
for all $t_n \ge2$. Since $F(\omega)$ is bounded in $L^\infty(\mcO
)$, by
Theorem \ref{thmgenerationRDS} $\vp(1,\break \t_{-1}\omega)F(\omega)$
is a set of
uniformly continuous functions on each compact set $K \subseteq\mcO$
with modulus of continuity depending only on $m$, $\operatorname{dist}(K,\partial\mcO
)$ and $\|F(\omega)\|_{L^\infty(\mcO)}$. Let $\{K_k| k \in\N\}$ be a
sequence of compact sets in $\mcO$, such that $\mcO= \bigcup_{k \in
\N
} K_k$. For each $k \in\N$ we can choose a convergent subsequence of
$\vp(t_n,\t_{-t_n}\omega)x_n \in C^0(K_k)$. Passing to a diagonal sequence,
we can thus choose a subsequence (again denoted by~$n$) such that $\vp
(t_n,\t_{-t_n}\omega)x_n$ is convergent in each $C^0(K_k)$ and in
particular pointwise convergent in all of $\mcO$. By the uniform
$L^\infty(\mcO)$ bound on $\vp(t_n;\t_{-t_n}\omega)x_n$ this implies
convergence of $\vp(t_n;\t_{-t_n}\omega)x_n$ in $L^p(\mcO)$, for
each $p
\in[1,\infty)$.

(ii): By Theorem \ref{thmgenerationRDS}(iii) the set $K(\omega
):=\vp
(1,\t_{-1}\omega)F(\omega)$ is uniformly bounded and equicontinuous
in $C^0(\bar
\mcO)$. Since $F(\omega)$ is absorbing, so is the set $\vp(1,\t
_{-1}\omega)F(\omega)$.
\end{pf}

\begin{pf*}{Proof of Theorem \ref{thmexistenceRA}}
Let $\mcD^p$ be the universe of all random sets in $L^p(\mcO)$, $p
\in
[1,\infty]$.

The (unique) existence of a $\mcD^p$-random attractor $A^p$ in
$L^p(\mcO)$ follows from $\mcD^p$-absorption, $\mcD^p$-asymptotic
compactness, quasi-weak-continuity of $\vp$ on $L^p(\mcO)$ and Theorem
\ref{thmexra} for each $p \in[1,\infty)$. Since $F$ in Proposition
\ref{propbddabs} is an $L^\infty$ bounded set absorbing all sets in
$\mcD^1$, all these attractors coincide.

By the invariance property of the random attractor and Proposition \ref
{propbddabs} we have $A(\omega) = \vp(t,\t_{-t}\omega)A(\t
_{-t}\omega) \subseteq
F(\omega)$, for all $t \ge1$ and thus $L^\infty$ boundedness of~$A$. Again
by invariance of~$A$, $A(\omega) = \vp(1,\t_{-1}\omega)A(\t
_{-1}\omega) \subseteq
\vp(1,\break \t_{-1}\omega)F(\t_{-1}\omega)$. Invoking Theorem \ref
{thmgenerationRDS} yields equicontinuity on each compact set
$K\subseteq\mcO$.

If $(\mcO1)$ is satisfied, then we can argue as above for $p = \infty$.
\end{pf*}

\section*{Acknowledgments}
The author would like to thank Michael R\"ockner for valuable discussions
and comments. Several helpful comments by the anonymous referee are
gratefully acknowledged.


%


\printaddresses

\end{document}